\documentclass[11pt]{article}

\usepackage{amsmath,amssymb,amsfonts,amsthm,amsbsy}
\usepackage{mathrsfs,bbm,bm,dsfont}
\usepackage{graphicx,float}
\usepackage{verbatim,fancyhdr}
\usepackage[hang]{caption}
\usepackage{cancel}
\usepackage{tikz}
\usepackage[dvipsnames]{xcolor}
\usepackage[authoryear]{natbib}
\usepackage[
  bookmarks,
  bookmarksnumbered,
  colorlinks=true,
  pdfstartview=FitV,
  linkcolor=Red,
  citecolor=blue!90!green,
  urlcolor=blue!90!green
]{hyperref}
\hypersetup{
  pdftitle={Asymptotics and finite-sample bounds for prediction and smoothing in Wright--Fisher hidden Markov models},
  pdfauthor={Luigi M. Malgieri, Filippo Ascolani, Matteo Giordano, Matteo Ruggiero}
}

\graphicspath{{figures/}}

\allowdisplaybreaks

\bibpunct{\textcolor{blue!90!green}{(}}{\textcolor{blue!90!green}{)}}{\textcolor{blue!90!green}{;}}{a}{\textcolor{blue!90!green}{,}}{\textcolor{blue!90!green}{;}}

\theoremstyle{plain}

\newtheorem{theorem}{Theorem}[section]
\newtheorem{lemma}[theorem]{Lemma}
\newtheorem{proposition}[theorem]{Proposition}
\newtheorem{corollary}[theorem]{Corollary}

\theoremstyle{definition}

\theoremstyle{remark}

\newcommand{\data}[1]{Y^{(n)}_{#1}}

\def\bs#1{\boldsymbol{#1}}
\def\aa{\bs\alpha}

\def\kk{\mathbf{k}}

\newcommand{\trans}[2]{p_{#1 | #2}}
\newcommand{\bridge}[3]{b_{#1 | #2,#3}}
\newcommand{\Mult}{\mathrm{Cat}}

\newcommand{\post}{q}

\def\nn{\mathbf{n}}

\newcommand{\Z}{\mathbb{Z}_{+}}

\newcommand{\TV}[2]{d_{TV}\!\left(#1, #2\right)}

\newcommand{\given}{|}

\newcommand{\prob}[1]{\mathbb{P}\!\left(#1\right)}
\newcommand{\mean}[1]{\mathbb{E}\left[#1\right]}

\newcommand{\sX}{\mathcal{X}}
\newcommand{\de}{\mathrm{d}}

\newcommand{\coal}[1]{p_{#1}(t)}
\newcommand{\ncoal}[2]{p_{#1, #2}(t)}

\newcommand{\Dirmeasvanilla}[1]{\pi_{\aa}}

\let\oldsum\sum
\renewcommand{\sum}{\oldsum\limits}
\let\oldprod\prod
\renewcommand{\prod}{\oldprod\limits}
\let\oldlim\lim
\renewcommand{\lim}{\oldlim\limits}
\let\oldsup\sup
\renewcommand{\sup}{\oldsup\limits}

\newcommand{\real}{\mathbb{R}}
\newcommand{\integer}{\mathbb{Z}}

\newcommand{\timegap}{h}
\newcommand{\maybeincludegraphics}[2][]{%
  \IfFileExists{#2}{%
    \includegraphics[#1]{#2}%
  }{%
    \fbox{\parbox[c][0.22\textheight][c]{0.45\textwidth}{\centering Missing figure\\\texttt{#2}}}%
  }%
}

\renewcommand{\epsilon}{\varepsilon}

\title{Asymptotics and finite-sample bounds for prediction and smoothing in Wright--Fisher hidden Markov models}
\author{Luigi M. Malgieri$^{1}$, Filippo Ascolani$^{1}$, Matteo Giordano$^{2}$ and Matteo Ruggiero$^{3}$\\[2mm]
\small $^{1}$Department of Statistical Science, Duke University, Durham, NC, USA\\
\small $^{2}$ESOMAS Department and Collegio Carlo Alberto, University of Torino, Torino, Italy\\
\small $^{3}$Stern School of Business, NYU Abu Dhabi, UAE}
\date{}
\begin{document}
\maketitle
\begin{abstract}
We study signal prediction and smoothing in hidden Markov models whose latent signal is a multi-type Wright--Fisher diffusion and whose discrete-time observations are categorical. We consider a repeated-sampling regime with finitely many inspection times and large within-time samples. Under neutrality and parent-inde\-pendent mutation, the Bayesian predictive and smoothing laws admit exact representations as mixtures of Dirichlet distributions. We show that the one-step predictive and local smoothing laws converge in total variation to the corresponding Wright--Fisher transition and bridge laws determined by the neighboring population frequencies, thereby justifying plug-in approximations based on empirical proportions. For prediction, we derive explicit finite-sample bounds for this approximation and for bounded predictive functionals such as diversity, together with a finite-sample concentration bound for the marginal smoother. The joint posterior at the inspection times concentrates at the target frequencies and satisfies a Bernstein--von Mises limit with distinct interior and boundary regimes, with Gaussian and Gamma limits on different scales. The results allow arbitrary dependence across inspection times. A key technical ingredient is a fixed-interval tail bound for the coalescent-based dual death process.
\end{abstract}
\tableofcontents

\section{Introduction}

Hidden Markov models (HMMs) provide a standard framework for inferring latent dynamics from noisy discrete-time observations, see, e.g., \cite{cappe2005}. An unobserved signal evolves according to a Markov process, while data are emitted conditionally on the current state. The primary inferential tasks are filtering and smoothing, namely reconstructing current and past or intermediate states from the observed data. 
These problems are well understood in linear-Gaussian and finite-state settings. Linear-Gaussian models admit exact Kalman and Kalman--Bucy recursions \citep{kalman1960,kalmanbucy1961}, while finite-state HMMs support a substantial asymptotic literature for likelihood-based inference and stability \citep{bickel1998,douc2004}. Many modern applications, however, involve latent states on other spaces, such as time-indexed frequency vectors, and therefore fall outside these standard frameworks.

The Wright--Fisher (WF) diffusion is the canonical model for evolving proportions because it lives on the simplex, preserves nonnegativity and unit mass, and has a rich probabilistic structure; see \citealp{ethier_kurtz_1986}, Chapter 10. WF-type dynamics have found applications in population genetics \citep{Bollback,Steinrucken}, econometrics \citep{Gourieroux}, mathematical finance \citep{Filipovic}, topic modeling \citep{Perrone}, Bayesian nonparametrics \citep{Mena}, opinion dynamics \citep{Toscani}, and neuroscience \citep{DOnofrio2}. In population genetics, for example, the latent object is an allele-frequency trajectory, while the observations consist of counts sampled from the population types at a moderate number of times \citep{Malaspinas2012,Schraiber}. Ancient genomic studies provide increasingly rich time-series data of this kind \citep{Wutke2016}.

Under neutrality and parent-independent mutation, the WF signal is Dirichlet-stationary and reversible, and its transition and bridge laws admit exact representations as infinite mixtures of Dirichlet distributions through duality with Kingman's coalescent; see Section~\ref{sec:wf-preliminaries}. This dual structure has been exploited to construct exact filtering, smoothing, and out-of-sample prediction procedures from noisy categorical samples \citep{papaspiliopoulos_optimal_2014,papaspiliopoulos_conjugacy_2016,ascolani_predictive_2021,ascolani_smoothing_2023}, together with associated software \citep{KKK21}. A related line of work instead exploits exact simulation from WF transition and bridge laws for inference \citep{jenkins2017exact,sant2023ewf,Sant25}.

The distinction between these approaches is in how the latent endpoint frequencies enter the analysis. Exact filtering and smoothing yield the Bayesian conditional laws under the WF--HMM by integrating over posterior uncertainty in the latent frequencies, whereas transition- and bridge-based methods use endpoint values directly, typically replacing the latent population frequencies by empirical proportions. A central question of this paper is therefore how accurate this plug-in replacement is. We show that the exact predictive and local smoothing laws converge to the corresponding WF transition and bridge laws as the endpoint frequencies become consistently estimated, and for prediction we quantify the finite-sample error incurred by using the empirical proportions.

We study this question in a repeated-sampling regime with fixed inspection times $0<t_1<\cdots<t_k$ and growing within-time sample sizes, assuming independence only within each inspection time. This regime is intended for settings with a small or moderate number of inspection times but substantial within-time information. Ancient-DNA studies provide one example, with observations aggregated into a limited number of temporal bins while genome-scale information accumulates within each bin \citep{He2023,Bouzid2021}. Allowing dependence across inspection times also accommodates longitudinal or panel sampling of categorical units \citep{Maruotti2011}. Under independence across inspection times, the first-order finite-grid asymptotics are close to a standard parametric Bayesian problem for the latent compositions. We show that the same transition, bridge, and joint posterior limits persist under arbitrary cross-time dependence in the data-generating array, and further obtain explicit finite-sample bounds and a joint Bernstein--von Mises theorem with distinct interior and boundary regimes.

Our first results concern prediction and marginal smoothing. Corollary~\ref{thm:predictive-empirical} bounds the error made by replacing the exact one-step predictive law with the WF transition law started from the empirical composition at the previous inspection time. The same bound yields explicit control of bounded predictive functionals, including diversity measures, through Corollary~\ref{cor:predictive-bounded-functional}. Theorem~\ref{cor:marginal-smoother-empirical} instead gives a finite-sample concentration bound for the full marginal smoother around the empirical composition at the current inspection time. We then identify the corresponding population limits. Corollary~\ref{prop:predictive-target} shows that the one-step predictive law converges in total variation to the WF transition law started from the population frequency at the preceding inspection time, while Proposition~\ref{thm:bridge-main} shows that the local smoothing law converges to the corresponding WF bridge law determined by the neighboring population frequencies. Proposition~\ref{cor:extended-locality} further shows that these local limits are unchanged by conditioning on additional non-adjacent inspection times.

Finally, Theorem~\ref{thm:consistency-bvm} establishes posterior concentration for the full joint law at the inspection times and a Bernstein--von Mises limit with distinct interior and boundary regimes. Interior coordinates exhibit the usual Gaussian approximation, while boundary coordinates fluctuate on the smaller $n^{-1}$ scale and converge to Gamma limits. When all true frequencies are positive, the boundary component disappears and the result reduces to a block-diagonal Gaussian approximation. Although related to classical finite-dimensional Bernstein--von Mises theory \citep[Chapter~10]{vaart_asymptotic_1998}, the result allows arbitrary dependence across distinct inspection times in the data-generating array. The exact posterior nevertheless becomes asymptotically local, in the sense that first-order uncertainty at each inspection time is determined by the corresponding within-time categorical sample.

The main technical input is a fixed-interval tail bound for the coalescent-based dual death process. Proposition~\ref{prop:coal-tail} provides uniform control of large dual counts and allows the transition and bridge mixtures to be truncated uniformly over the data.

The paper is organized as follows. Section~\ref{sec:preliminaries} collects the WF-HMM preliminaries, including the transition, bridge, and exact finite-sample conditional laws. Section~\ref{sec:results} formalizes the repeated-sampling regime and presents the prediction, smoothing, and joint posterior results. Section~\ref{sec:numerics} provides numerical illustrations of the predictive and local smoothing approximations and of the finite-sample predictive bounds. Section~\ref{sec:conclusion} discusses implications and limitations. The Supplementary Material contains the proofs, auxiliary results, and additional numerical illustrations.


\section{Wright--Fisher HMMs and sampling framework}\label{sec:preliminaries}

This section introduces the WF hidden Markov model and the exact conditional laws studied in the paper. 
For the statistical developments that follow, the key point is that these Bayesian conditional laws are explicit Dirichlet mixtures, so the asymptotic problem becomes one of controlling the associated mixing weights. The data generating mechanism under which we derive our (frequentist) asymptotic results and finite-sample bounds is also described.

\subsection{The Wright--Fisher diffusion}\label{sec:wf-preliminaries}

For $d\ge2$, let
\[
\Delta_d=\left\{x\in[0,1]^d:\sum_{r=1}^d x_r=1\right\}
\]
be the probability simplex, and let
$\aa=(\alpha_1,\ldots,\alpha_d)\in(0,\infty)^d$, with
$\theta=\sum\nolimits_{r=1}^d\alpha_r$. We write $X=\{X_t\}_{t\ge0}$ for the neutral $d$-type WF diffusion with parent-independent mutation parameter $\aa$, which is the $\Delta_d$-valued diffusion with infinitesimal generator
\begin{equation}\label{eq:wf-generator}
\mathcal{A}f(x)
=
\frac{1}{2}
\sum_{j=1}^d
\bigl(\alpha_j-\theta x_j\bigr)
\frac{\partial f}{\partial x_j}(x)
+
\frac{1}{2}
\sum_{j,\ell=1}^d
x_j(\delta_{j\ell}-x_\ell)
\frac{\partial^2 f}{\partial x_j\partial x_\ell}(x),
\end{equation}
for sufficiently regular $f$. We use the generator formulation, which is standard for simplex-valued WF diffusions and is the natural one for the semigroup and duality arguments below. The process is reversible with respect to the Dirichlet distribution $\pi_{\aa}$, which is also its stationary distribution; see, e.g., \citet[Chs.~10--11]{ethier_kurtz_1986}. We take this stationary WF process as the prior model for the latent signal.

The transition density of the WF diffusion is not available in a simple finite-dimensional closed form, but several exact representations are known. Under parent-independent mutation, one particularly useful representation is a Dirichlet mixture arising from duality with Kingman's coalescent \citep{kingman1982coalescent}. Roughly speaking, this duality expresses moments of the forward WF diffusion through a backward ancestral counting process, thereby representing the transition law as a mixture over latent ancestral configurations; see also \citet{papaspiliopoulos_optimal_2014}. At a heuristic level, over a time gap $s$ one first draws an auxiliary count $k\ge0$, interpretable as the number of ancestral lineages surviving backward in time, then allocates these $k$ lineages among the $d$ types according to the starting composition $x$, and finally samples the arrival state from a Dirichlet law with parameter $\aa+\kk$. This gives the exact transition density
\begin{equation}\label{eq:wf-transition-compact-sec2}
p_s(x,x')
=
\sum_{\kk\in\mathbb Z_+^d}
p_{\kk}^x(s)\,\pi_{\aa+\kk}(x'),
\qquad x,x'\in\Delta_d,
\end{equation}
where
\[
p_{\kk}^x(s)
:=
p_{|\kk|}(s)
\binom{|\kk|}{\kk}
x^{\kk},
\qquad
x^{\kk}:=\prod_{r=1}^d x_r^{k_r}.
\]
Here $\{p_k(s)\}_{k\ge0}$ are the entrance-law probabilities of the coalescent-based dual death process over a time interval of length $s$, while $\pi_{\aa+\kk}$ denotes the Dirichlet density with parameter $\aa+\kk$. See \citet{griffiths1979transition} for the WF transition representation and \citet{tavare_1984} for the dual death-process weights. These mixture weights are the population quantities that will later arise as asymptotic limits of the finite-sample predictive laws.

Given $0<t<s<u$ and endpoints $x,x'\in\Delta_d$, the WF bridge law at time $s$ is the conditional law of the latent composition at the intermediate time $s$, given that the process is at $x$ at time $t$ and at $x'$ at time $u$. 
Under parent-independent mutation, this bridge law is again an explicit mixture of Dirichlet distributions, given by
\begin{equation}\label{eq:wf-bridge-mixture-sec2}
\bridge{s}{t}{u}(z\mid x,x')
=\sum_{\kk\in\mathbb Z_+^d}\sum_{\kk'\in\mathbb Z_+^d}
w_{\kk,\kk'}(s-t,u-s)\,\pi_{\aa+\kk+\kk'}(z),
\end{equation}
where
\begin{equation}\label{eq:wf-bridge-weights-sec2}
w_{\kk,\kk'}(h,h')
\propto
C_{\kk,\kk'}\,p_{\kk}^{x}(h)\,p_{\kk'}^{x'}(h'),
\end{equation}
and $C_{\kk,\kk'}$ is the explicit coefficient arising from the product rule for Dirichlet densities,
$$\pi_{\aa+\kk}(z)\pi_{\aa+\kk'}(z)/\pi_{\aa}(z) \propto C_{\kk,\kk'}\,\pi_{\aa+\kk+\kk'}(z);$$ see equation~\eqref{eq:wf-bridge-C-sec2} of the Supplementary Material. Thus the bridge combines a left and a right ancestral count vector, one propagated from each endpoint, and merges the corresponding Dirichlet components through the coefficient $C_{\kk,\kk'}$. For two-type bridge formulas and genealogical interpretations, see \citet{griffiths_wrightfisher_2018}. This bridge mixture is the population law that later appears as the asymptotic target of the local smoothing distributions.

\subsection{The Wright--Fisher hidden Markov model}
\label{sec:hmm-posteriors}

Fix an inspection grid $0<t_1<\cdots<t_k$ and let $X_i:=X_{t_i}$ denote the latent diffusion state at time $i$. Conditionally on $X_i=x\in\Delta_d$, the observations collected at time $i$ are categorical with parameter $x$:
\begin{equation}\label{eq:observation-model-sec2}
Y_{i,j}\mid X_i \stackrel{\text{i.i.d.}}{\sim} \mathrm{Cat}(X_i),
\qquad i=1,\ldots,k,\quad j=1,\ldots,n_i.
\end{equation}
Equivalently, if $\nn_i=(n_{i,1},\ldots,n_{i,d})$ denotes the vector of category counts at time $t_i$, then $\nn_i\mid X_i=x$ is multinomial with size $n_i$ and cell probabilities $x$.  The Bayesian model is completed by taking the law of a Wright--Fisher process with transition density \eqref{eq:wf-transition-compact-sec2} as the prior for the latent signal $X$.


\subsection{Data-generating mechanism}
\label{sec:true}

For each inspection time $i=1,\ldots,k$, let $x_i^*\in\Delta_d$ denote the ``true'' category proportions at time $i$. Let
\[
\bigl(Y_{i,j}:i=1,\ldots,k,\ j\ge1\bigr)
\]
be defined on a common probability space with joint law $\mu_{1:k}^{(\infty)}$. For each $i$,
\begin{equation}\label{eq:true_mechanism}
Y_{i,1},Y_{i,2},\ldots
\overset{\mathrm{i.i.d.}}{\sim}
\Mult(x_i^*),
\end{equation}
where $\Mult(x_i^*)$ denotes the categorical law on $\{1,\ldots,d\}$ with probability vector $x_i^*$. Dependence across distinct inspection times is otherwise unrestricted. For each $n$, let $\mu_{1:k}^{(n)}$ denote the marginal law of the first $n$ observations at each inspection time: we only require that they are a projection of a distribution over infinite sequences $\mu_{1:k}^{(\infty)}$.

Let $Y_i^{(n)}=(Y_{i,1},\ldots,Y_{i,n})$ denote the sample at time $i$, and $Y_{i:j}^{(n)}=(Y_i^{(n)},\ldots,Y_j^{(n)})$ the samples from times $i$ through $j$. Let $\mu_i$ denote the law of $Y_{i,1}$, and write $\mu_i^{(n)}=\mu_i^{\otimes n}$ and $\mu_i^{(\infty)}=\mu_i^{\otimes\infty}$. For $1\le i\le j\le k$, let $\mu_{i:j}^{(n)}$ and $\mu_{i:j}^{(\infty)}$ denote the corresponding joint laws across inspection times; these need not factorize across times. For future reference, set
\begin{equation}\label{eq:proportions}
\hat x_i
=
\Big(
\frac{n_{i,1}}{n},
\ldots,
\frac{n_{i,d}}{n}
\Big),
\qquad
\hat x_{1:k}
=
(\hat x_1,\ldots,\hat x_k),
\end{equation}
where $n_{i,j}$ is the multiplicity of category $j$ in $Y_i^{(n)}$.


\subsection{Exact finite-sample predictive and smoothing laws}
\label{sec:local-finite-sample-formulas}

The inferential objects studied in this paper are the conditional laws of latent states under different conditioning sets, corresponding to prediction, smoothing, and joint latent-state reconstruction.
In particular, we consider the following posterior laws:
\begin{list}{
$\bullet$
}{\itemsep=1mm\topsep=2mm\itemindent=-3mm\labelsep=2mm\labelwidth=0mm\leftmargin=5mm\listparindent=0mm\parsep=0mm\parskip=0mm\partopsep=0mm\rightmargin=0mm\usecounter{enumi}}
\setcounter{enumi}{0}
\item the one-step predictive posterior $\post_{i\mid i-1}(\cdot\mid \data{i-1})$ of $X_i$ given the sample at time $i-1$;
\item the local smoother posterior $\post_{i\mid i-1,i+1}(\cdot\mid \data{i-1},\data{i+1})$ of $X_i$ given the samples at times $i-1$ and $i+1$;
\item the enlarged-conditioning analogues of the two above laws,  $\post_{i\mid 1:i-1}(\cdot\mid \data{1:i-1})$ and $\post_{i\mid 1:i-1,i+1:k}(\cdot\mid \data{1:i-1},\data{i+1:k})$;
\item the full marginal smoother $\post_{i\mid 1:k}(\cdot\mid \data{1:k})$, i.e.\ the marginal posterior of $X_i$ under all observations;
\item the joint smoother $\post_{1:k\mid 1:k}(\cdot\mid \data{1:k})$, i.e.\ the posterior of $(X_1,\ldots,X_k)$ under all observations.
\end{list}
A key consequence of the coalescent duality recalled in Section \ref{sec:wf-preliminaries} is that the above conditional laws can still be written explicitly as finite mixtures of Dirichlet distributions whose weights are determined by finite coalescent count configurations.
Specifically, for $h_i=t_i-t_{i-1}$, the one-step predictive density admits the finite mixture representation
\begin{equation}\label{eq:predictive-density-sec2}
\post_{i\mid i-1}(x\mid \data{i-1})
=\sum_{\kk\le\nn_{i-1}} p_{\nn_{i-1},\kk}(h_i)\,\pi_{\aa+\kk}(x),
\qquad x\in\Delta_d,
\end{equation}
where
\[
p_{\nn, \kk}(s):=p_{|\nn|,|\kk|}(s)\binom{n}{k}^{-1}\binom{\nn}{\kk}.
\]
Here $\{p_{n,k}(s)\}_{k \geq 0}$ are the transition probabilities of a suitable death process starting from $n$; see \citet{papaspiliopoulos_optimal_2014} and \citet[Lemma~4.1]{papaspiliopoulos_conjugacy_2016}. Thus the population transition mixture \eqref{eq:wf-transition-compact-sec2} is replaced, at finite sample size, by a finite Dirichlet mixture indexed by all surviving count vectors $\kk\le\nn_{i-1}$.

Similarly, the local smoother admits the representation
\begin{equation}\label{eq:smoother-sec2}
\post_{i\mid i-1,i+1}(x\mid \data{i-1},\data{i+1})
=\sum_{\kk\le\nn_{i-1},\,  \kk'\le\nn_{i+1}}
w^{\nn_{i-1},\nn_{i+1}}_{\kk,\kk'}(h_i,h_{i+1})\,
\pi_{\aa+\kk+\kk'}(x),
\end{equation}
where the weights couple the forward and backward coalescent contributions through
\begin{equation}\label{smoother-weights}
w^{\nn_{i-1},\nn_{i+1}}_{\kk,\kk'}(h_i,h_{i+1})
\propto
C_{\kk,\kk'}\,
p_{\nn_{i-1},\kk}(h_i)\,
p_{\nn_{i+1},\kk'}(h_{i+1}),
\end{equation}
with $C_{\kk,\kk'}$ as in Supplementary Material, equation~\eqref{eq:wf-bridge-C-sec2}; see \citet{KKK21}. This is the finite-sample analogue of the population bridge mixture \eqref{eq:wf-bridge-mixture-sec2}.

The importance of \eqref{eq:predictive-density-sec2} and \eqref{eq:smoother-sec2} is that they make the predictive and smoothing laws fully explicit at finite sample size. This makes it possible to derive quantitative finite-sample bounds and to identify the corresponding population limits by comparing these Dirichlet mixtures with the transition and bridge mixtures in \eqref{eq:wf-transition-compact-sec2} and \eqref{eq:wf-bridge-mixture-sec2}. Moreover, by combining \eqref{eq:predictive-density-sec2} and \eqref{eq:smoother-sec2}, it is possible to express all the other posterior laws of interest as suitable finite mixtures of Dirichlet laws. 


\section{Finite-sample approximations and asymptotic results}\label{sec:results}

\subsection{A fixed-interval tail bound for the dual death process}\label{sec:probability_tails}

The analytical validation of WF filtering and smoothing requires passing from finite-sample Dirichlet mixtures to their infinite-population limits.
As discussed in Section \ref{sec:wf-preliminaries}, both the transition density \eqref{eq:wf-transition-compact-sec2} and the bridge density \eqref{eq:wf-bridge-mixture-sec2} are infinite mixtures where the index $k$ represents the number of ancestral lineages surviving over a time interval $t$. The corresponding weights depend on $p_k(t)=\mathbb{P}(N_t=k \mid N_0 = \infty)$, where $N_t$ is the block-counting process of Kingman's coalescent. Similarly, the predictive and smoothing distributions \eqref{eq:predictive-density-sec2} and \eqref{eq:smoother-sec2} depend on $p_{n,k}(t)=\mathbb{P}(N_t=k \mid N_0 =n)$, where the Kingman's coalescent is started at $n$, the total number of observations collected at each time.

Therefore, a key tool in establishing total variation convergence for these laws is to provide uniform control over the mixture tails. Specifically, we require an explicit bound on the remainder
\begin{equation}\label{eq:remainder-def}
\sum_{k>M} p_{n,k}(t) = \mathbb{P}(N_t > M \mid N_0 = n),
\end{equation}
which quantifies the truncation error incurred by keeping only terms with at most $M$ lineages. 
While the exact density of $N_t$ is known \citep{tavare_1984}, its series form with alternating signs is not convenient for the arguments used below. The following upper bound is based on Markov's inequality and the representation of the Kingman's coalescent thorugh sums of exponential distributions.

\begin{proposition}\label{prop:coal-tail}
For any \(t>0\), {there exist constants \(c>0\) and $R > 0$} such that for every $n \in \mathbb{N}\cup \{+\infty\}$ and \(M\ge 0\),
\[
\sum_{k>M}p_{n,k}(t)\ \le\ {R}\exp\{-cM^2 t\}.
\]
In particular, {\(p_{n,k}(t)\le R\exp\{-c(k-1)^2 t\}\)} for all \(k\).
\end{proposition}

The proof is given in Supplementary Material, Section~\ref{app:truncation}.
This bound is the main technical input behind the results that follow. Small-time lineage asymptotics and large-deviation results for related pure-death models are available in the literature; see, for example, \citet{griffiths_asymptotic_1984,
BansayeMeleardRichard2016}. However, existing estimates do not directly provide fixed-interval truncation control needed for the WF transition and bridge mixtures.

\subsection{Prediction}
\label{sec:prediction}

We begin with a finite-sample comparison between the exact one-step predictive law and a WF transition started from an arbitrary point $x\in\Delta_d$. This yields, as special cases, comparisons with the transition started from the population target $x_{i-1}^*$ and from the observable empirical composition $\hat x_{i-1}$.

\begin{theorem}\label{thm:predictive-rate}
For every $i=2,\dots,k$ and $x\in\Delta_d$,
\[
\TV{\post_{i\mid i-1}(\cdot\mid \data{i-1})}
{\trans{i}{i-1}(\cdot\mid x)}
\le
L_{h_i}\left(
\|\hat x_{i-1}-x\|_1
+\frac{2\theta}{n+\theta}
+\sqrt{\frac{d}{n+\theta+1}}
\right)
\]
almost surely under $\mu_{i-1}^{(n)}$, where $h_i=t_i-t_{i-1}$ and $L_{h_i}$ is the constant from Lemma~\ref{lem:transition-lipschitz}. In particular, for every $r>0$,
\[
\begin{aligned}
&\mu_{i-1}^{(n)}\!\left(
\TV{\post_{i\mid i-1}(\cdot\mid \data{i-1})}
{\trans{i}{i-1}(\cdot\mid x_{i-1}^*)}
>
L_{h_i}\left(
r+\frac{2\theta}{n+\theta}
+\sqrt{\frac{d}{n+\theta+1}}
\right)
\right)
\\[-1mm]
&\qquad\le
2d\exp\!\left\{-\frac{2nr^2}{d^2}\right\}.
\end{aligned}
\]
\end{theorem}

Theorem~\ref{thm:predictive-rate} separates the error into the mismatch between the chosen starting point $x$ and the empirical composition $\hat x_{i-1}$, and a finite-sample term of order $n^{-1/2}$. 
Taking $x=\hat x_{i-1}$ in Theorem~\ref{thm:predictive-rate} eliminates the discrepancy term $\|\hat x_{i-1}-x\|_1$ and yields the following approximation based on observed proportions.

\begin{corollary}\label{thm:predictive-empirical}
For every $i=2,\dots,k$,
\[
\TV{\post_{i\mid i-1}(\cdot\mid \data{i-1})}
{\trans{i}{i-1}(\cdot\mid \hat x_{i-1})}
\le
\varepsilon_{i,n}:=
L_{h_i}\left(
\frac{2\theta}{n+\theta}
+\sqrt{\frac{d}{n+\theta+1}}
\right)
\]
almost surely under $\mu_{i-1}^{(n)}$.
\end{corollary}

The constant can be taken as \(L_h=\sum\nolimits_{k\geq0}k\,p_k(h)\), the expected number of ancestral lineages at time \(h\) under the coalescent entrance law; see Lemma~\ref{lem:transition-lipschitz} of the SM. It is finite for every \(h>0\) and diverges as \(2/h\) as \(h\downarrow0\), consistent with the small-time asymptotics of Kingman's coalescent with mutation (see \citealp{griffiths_asymptotic_1984}). Thus, the bound is naturally interpreted for fixed positive inspection gaps.

The bound is most informative when the within-time sample size is large relative to the simplex dimension and the inspection gap is not too small. Its deterioration as $h_i$ decreases reflects the fact that the exact predictive law $\post_{i\mid i-1}(\cdot\mid \data{i-1})$ averages over posterior uncertainty about $X_{i-1}$, whereas the approximation starts the WF transition from the single point $\hat x_{i-1}$. As $h_i\downarrow0$, the WF semigroup has less time to smooth this discrepancy. At $h_i=0$, the exact predictive law is $\pi_{\aa+\nn_{i-1}}$ while the plug-in law is $\delta_{\hat x_{i-1}}$, so their total-variation distance equals one for every fixed $n$. 
The additional numerical comparisons in Section~\ref{app:numerical-illustrations} of the SM are consistent with this dependence on the inspection gap, with larger gaps generally yielding smaller predictive discrepancies.

The preceding bound also immediately controls bounded predictive functionals. 

\begin{corollary}\label{cor:predictive-bounded-functional}
For every $i=2,\dots,k$ and every bounded measurable function
$f:\Delta_d\to\real$,
\[
\left|
\int_{\Delta_d} f(z)\,\post_{i\mid i-1}(z\mid \data{i-1})\,\de z
-
\int_{\Delta_d} f(z)\,\trans{i}{i-1}(z\mid \hat x_{i-1})\,\de z
\right|
\le
2||f||_{\infty}\varepsilon_{i,n},
\]
with $\varepsilon_{i,n}$ as in Corollary~\ref{thm:predictive-empirical}.
\end{corollary}

As a simple illustration, consider Simpson's diversity index
\begin{equation}\label{diversity index}
f(x)=1-\sum_{j=1}^d x_j^2,
\qquad x\in\Delta_d.
\end{equation}
Since $0\le f(x)\le1$ on $\Delta_d$, Corollary~\ref{cor:predictive-bounded-functional} gives
\begin{equation}\label{eq:simpson-bound}
\left|
\mathbb E_{\post_{i\mid i-1}(\cdot\mid \data{i-1})}
\!\left[1-\sum_{j=1}^d X_{i,j}^2\right]
-
\mathbb E_{\trans{i}{i-1}(\cdot\mid \hat x_{i-1})}
\!\left[1-\sum_{j=1}^d X_{i,j}^2\right]
\right|
\le
2\varepsilon_{i,n}.
\end{equation} 

Thus the predictive expectation of diversity under the exact filter can be replaced by the corresponding expectation under the WF transition started from the observed composition $\hat x_{i-1}$, with explicit worst-case error control. In practice, this expectation may be approximated by Monte Carlo using exact WF simulation methods; see, for example, \citet{jenkins2017exact,sant2023ewf}. Section~\ref{sec:numerics} illustrates both the law-level predictive approximation in Corollary~\ref{thm:predictive-empirical} and its consequence for the bounded Simpson-diversity functional in Corollary~\ref{cor:predictive-bounded-functional}. For several values of the population frequency $x^*$, the empirical total-variation discrepancy and the corresponding predictive mean gaps decrease with $n$ on the same qualitative scale as the $n$-dependent factor in the bound. Section~\ref{app:numerical-illustrations} of the SM reports the analogous rate comparisons for additional inspection gaps, together with further numerical checks.

We conclude this section by identifying the population target behind the finite-sample approximation. By combining Corollary~\ref{thm:predictive-empirical} with the convergence of $\hat x_{i-1}$ to $x_{i-1}^*$, the next corollary shows that the predictive distribution converges to the one-step WF transition started from the population value $x_{i-1}^*$.

\begin{corollary}\label{prop:predictive-target}
Let $Y_{1:k}^{(n)}$ be generated as in \eqref{eq:true_mechanism}. Then, for every $i=2,\ldots,k$,
\begin{equation}\label{eq:filter-tv}
d_{\mathrm{TV}}\!\big(\post_{i\mid i-1}(\cdot\mid \data{i-1}),
\trans{i}{i-1}(\cdot\mid x_{i-1}^*)\big)\ \to\ 0,
\end{equation}
$\mu_{i-1}^{(\infty)}$-almost surely as $n\to\infty$.
\end{corollary}

The corresponding high-probability bound is given in the second part of Theorem~\ref{thm:predictive-rate}. Corollary~\ref{prop:predictive-target} could also be obtained directly from \eqref{eq:predictive-density-sec2}, since by Lemmas~\ref{lm:tech_convergence} and~\ref{lemma: coalescent density convergence} of the Supplementary Material, the mixture weights converge termwise to those of $\trans{i}{i-1}(\cdot\mid x_{i-1}^*)$ in \eqref{eq:wf-transition-compact-sec2}, and since both are probability vectors, Scheff\'e's lemma yields convergence in total variation. The role of Corollary~\ref{thm:predictive-empirical} is instead to make this approximation quantitative at finite $n$. For the smoothing laws, whose weights are normalized and involve the factors $C_{\kk,\kk'}$ in \eqref{smoother-weights}, termwise convergence is not sufficient on its own, and the uniform tail control of Proposition~\ref{prop:coal-tail} is required.


\subsection{Marginal smoothing}
\label{sec:marginal-smoothing}

The next proposition is the smoothing analogue of Corollary~\ref{prop:predictive-target}, where observations are collected at both adjacent time points.
\begin{proposition}\label{thm:bridge-main}
Let $Y_{1:k}^{(n)}$ be generated as in \eqref{eq:true_mechanism}. Then, for every $i=2,\ldots,k-1$,
\begin{equation}\label{eq:bridge-tv}
d_{\mathrm{TV}}\!\big(\post_{i\mid i-1,i+1}(\cdot\mid \data{i-1},\data{i+1}),
\bridge{i}{i-1}{i+1}(\cdot\mid x_{i-1}^*,x_{i+1}^*)\big)\ \to\ 0,
\end{equation}
$\mu_{i-1,i+1}^{(\infty)}$-almost surely as $n\to\infty$.
\end{proposition}

Coherently with the previous results,  asymptotically the adjacent inspection times behave as if their endpoint frequencies were known, so the uncertainty remaining at time $i$ is exactly that of a WF bridge between $x_{i-1}^*$ and $x_{i+1}^*$. The statement extends immediately to heterogeneous sample sizes by replacing $n\to\infty$ with the local conditions $n_{i-1}\to\infty$ and $n_{i+1}\to\infty$.

Next we consider the full marginal smoother $\post_{i\mid 1:k}(\cdot\mid \data{1:k})$ and prove finite-sample concentration in a neighborhood $\hat{x}_i$ of radius $n^{-1/2}$.

\begin{theorem}\label{cor:marginal-smoother-empirical}
For every $i=2,\dots,k-1$, there exists a constant $A_{h_i, h_{i+1}}$ such that
\begin{equation}\label{marginal-smoother-empirical}
\post_{i\mid 1:k}\!\left(\left\{z\in\Delta_d:\|z-\hat x_i\|_2\ge \frac{r}{\sqrt{n}}\right\}\mid \data{1:k}\right)
\le
\frac{A_{h_i,h_{i+1}}}{n}+\frac{16}{r^2},
\end{equation}
for every $r >0$ and for every $n > N$, with {$N := N(\aa, r, h_i, h_{i+1})$}.
\end{theorem}

The bound in \eqref{marginal-smoother-empirical} contains two contributions, namely the term $A_{h_i,h_{i+1}}/n$, arising from the control of large coalescent counts, and the Dirichlet concentration term $16/r^2$. Thus Theorem~\ref{cor:marginal-smoother-empirical} provides finite-sample control of the full marginal smoother on the $n^{-1/2}$ scale. The corresponding posterior consistency statement is established, in greater generality, by the joint result in Section~\ref{sec:consistency}.


\subsection{Stability under additional conditioning}
\label{sec:extended-locality}

We record a secondary but structurally useful result, which states that enlarging the conditioning set beyond the neighboring inspection times does not change the first-order local targets. In particular, the conditional law of $X_i$ given all past data still has $\trans{i}{i-1}(\cdot\mid x_{i-1}^*)$ as its asymptotic target, and the conditional law of $X_i$ given all data except at time $i$ still has the local bridge $\bridge{i}{i-1}{i+1}(\cdot\mid x_{i-1}^*,x_{i+1}^*)$ as its asymptotic target. Its proof is deferred to Supplementary Material, Section~\ref{app:extended-locality}.

\begin{proposition}\label{cor:extended-locality}
For every $i=2,\ldots,k$, let $\post_{i\mid 1:i-1}(\cdot\mid \data{1:i-1})$ denote the conditional law of $X_i$ given all past data. Then
\begin{equation}\label{eq:filter-fullpast-tv}
d_{\mathrm{TV}}\!\big(\post_{i\mid 1:i-1}(\cdot\mid \data{1:i-1}),
\trans{i}{i-1}(\cdot\mid x_{i-1}^*)\big)\to 0,
\end{equation}
$\mu_{1:i-1}^{(\infty)}$-almost surely as $n\to\infty$.

Moreover, for every $i=2,\ldots,k-1$, let $\post_{i\mid 1:i-1,i+1:k}(\cdot\mid \data{1:i-1},\data{i+1:k})$ denote the conditional law of $X_i$ given all data except at time $i$. Then
\begin{equation}\label{eq:bridge-allother-tv}
d_{\mathrm{TV}}\!\big(\post_{i\mid 1:i-1,i+1:k}(\cdot\mid \data{1:i-1},\data{i+1:k}),
\bridge{i}{i-1}{i+1}(\cdot\mid x_{i-1}^*,x_{i+1}^*)\big)\to 0,
\end{equation}
$\mu_{1:i-1,i+1:k}^{(\infty)}$-almost surely as $n\to\infty$.
\end{proposition}

\subsection{Joint posterior asymptotics}
\label{sec:consistency}

We now turn to the joint posterior law of the inspection-time states $X_{1:k}$. At finite $n$, this law is built by combining information from earlier and later inspection times through forward--backward Dirichlet-mixture recursions. In particular, each inspection-time posterior remains a finite Dirichlet mixture, and the dependence across times enters only through the corresponding coalescent-based mixture weights.

The main joint result identifies the asymptotic behaviour of this posterior law at the inspection times. We retain a separate consistency statement because it remains valid without any interiority assumption and is proved in the stronger almost-sure form. The asymptotic fluctuations are shown to depend on the support of the true frequency vectors $x_i^*$. Coordinates corresponding to positive true frequencies satisfy the same Gaussian Bernstein--von Mises behaviour of the static Dirichlet-Multinomial model, whereas coordinates corresponding to zero true frequencies fluctuate on the smaller $n^{-1}$ scale and converge to Gamma limits. We write ``$\Rightarrow$'' for convergence in distribution.

\begin{theorem}\label{thm:consistency-bvm}
Let $Y_{1:k}^{(n)}$ be generated as in \eqref{eq:true_mechanism}. Then, as $n\to\infty$,
\begin{equation}\label{joint-consistency}
\post_{1:k \mid 1:k}(\cdot \mid \data{1:k})\ \Rightarrow\
\delta_{x_1^*}(\cdot)\times\cdots\times\delta_{x_k^*}(\cdot),
\end{equation}
$\mu_{1:k}^{(\infty)}$-almost surely.
Moreover, suppose, without loss of generality, that for each $i=1,\ldots,k$, for some $1\le r_i\le d$,
\[
x_i^*
=
(x_{i,1}^*,\ldots,x_{i,r_i}^*,0,\ldots,0),
\qquad
x_{i,j}^*>0,\quad j=1,\ldots,r_i.
\]
Write $X_i=(X_i^+,X_i^0)$, where $X_i^+=(X_{i,1},\ldots,X_{i,r_i})$ and $X_i^0=(X_{i,r_i+1},\ldots,X_{i,d})$, and define similarly $\hat x_i=(\hat x_i^+,\hat x_i^0)$. Let
\[
\Sigma_i^+=\operatorname{diag}(x_i^{*+})-x_i^{*+}x_i^{*+\top},
\qquad
\Sigma^+=\operatorname{diag}(\Sigma_1^+,\ldots,\Sigma_k^+).
\]
Then, as $n\to\infty$, in $\mu_{1:k}^{(\infty)}$-probability,
\begin{equation}\label{active-BvM}
\mathcal L\big((\sqrt n(X_i^+-\hat x_i^+))_{i=1}^k
\mid \data{1:k}\big)
\Rightarrow
\mathcal N(\cdot;0,\Sigma^+)
\end{equation}
and
\begin{equation}\label{boundary-gamma}
\mathcal L\big((nX_i^0)_{i=1}^k\mid \data{1:k}\big)
\Rightarrow
\bigotimes_{i=1}^k
\bigotimes_{j=r_i+1}^{d}
\Gamma(\alpha_j,1),
\end{equation}
with the convention that the inner product for each $i$ is a point mass at the empty vector when $r_i=d$. In fact, the convergences in \eqref{active-BvM} and \eqref{boundary-gamma} hold jointly, with limiting law equal to the product of the laws in \eqref{active-BvM} and \eqref{boundary-gamma}. If $r_i=d$ for all $i=1,\dots,k$, then \eqref{active-BvM} holds also in total variation.
\end{theorem}

A key feature of Theorem~\ref{thm:consistency-bvm} is that it imposes no assumption on dependence across distinct inspection times. The only sampling assumption is within-time independence, so the conclusion applies even when the observations are longitudinally dependent across inspection times. The theorem shows that, after integrating over the full forward--backward HMM structure, the exact joint posterior behaves to first order as if each inspection time were estimated separately from its own multinomial sample. For coordinates corresponding to positive true frequencies, this yields a block-diagonal Gaussian limit, so posterior correlations across distinct inspection times are asymptotically negligible. For coordinates corresponding to zero true frequencies, the posterior fluctuations occur on the smaller $n^{-1}$ scale rather than the usual $n^{-1/2}$ scale, and the corresponding rescaled coordinates converge to Gamma laws with shape parameters determined by the mutation parameters $\alpha_j$. Thus the Gaussian approximation survives on the active face of the simplex, while the inactive coordinates exhibit a distinct boundary asymptotic regime. In the interior case $r_i=d$ for every $i$, the inactive block is empty and Theorem~\ref{thm:consistency-bvm} recovers the standard Gaussian Bernstein--von Mises approximation of the static Dirichlet-Multinomial model.

The proof has three steps. First, one rewrites the relevant local conditional laws as Dirichlet mixtures indexed by neighbouring latent states and proves a uniform truncation bound for their mixture weights. Second, one combines this truncation with the usual concentration of Dirichlet laws to obtain local conditional concentration around the inspection-time target frequencies, which yields \eqref{joint-consistency}. Third, one studies the asymptotic behaviour of the Dirichlet components themselves. For coordinates corresponding to positive true frequencies, the usual finite-dimensional Bernstein--von Mises theorem for regular statistical models yields the Gaussian approximation \eqref{active-BvM}. For coordinates corresponding to zero true frequencies, the appropriate scaling is $n$ rather than $\sqrt n$, and the corresponding posterior fluctuations converge to Gamma limits, giving \eqref{boundary-gamma}. The detailed intermediate statements and proofs are collected in Supplementary Material, Section~\ref{app:bvm-aux}.


\section{Numerical illustrations}
\label{sec:numerics}

We give a brief numerical illustration of the results, focusing on four representative displays, and defer additional experiments to Section~\ref{app:numerical-illustrations} of the SM. The supplementary experiments include a more moderate predictive comparison, predictive rate comparisons across inspection gaps, the Gaussian approximation at an inspection time, a three-type boundary regime, and sensitivity to the mutation parameter and inspection gap. Throughout this section, transition and bridge quantities involving the coalescent representation are evaluated using a finite coalescent starting level of $1000$. Section~\ref{SM:coalescent-entrance} of the SM describes this finite-start approximation and reports representative sensitivity checks.

We first consider a two-type Wright--Fisher model with symmetric parent-independent mutation parameter $\aa=(2,2)$, near-boundary true frequency $x^*=0.05$, and inspection gap $h=0.01$. For each $n\in\{20,50,200\}$, a single dataset is generated under multinomial sampling at the previous inspection time, the empirical proportion is denoted by $\hat x$, and the exact predictive law $\post_{i\mid i-1}(\cdot\mid \data{i-1})$ is compared with the transition law $\trans{i}{i-1}(\cdot\mid \hat x)$ started at the empirical composition. This is the numerical experiment naturally associated with Corollary~\ref{thm:predictive-empirical}. The predictive curves are obtained by direct evaluation of the exact Beta-mixture formulas. Figure~\ref{fig:predictive-transition-harsh} shows that the discrepancy is clearly visible at small and moderate sample sizes and then decreases steadily as $n$ grows.

\begin{figure}[t!]
\centering
\includegraphics[width=\textwidth]{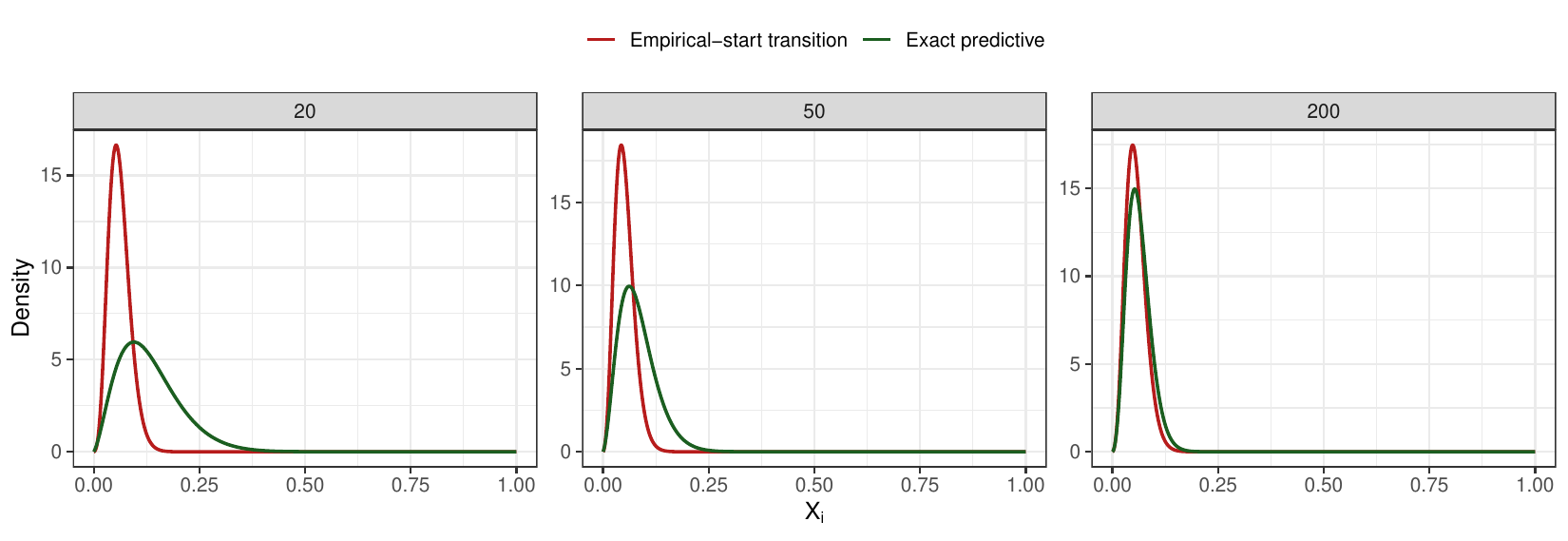}
\caption{\small Two-type predictive approximation with $\aa=(2,2)$, $x^*=0.05$, $h=0.01$, and $n\in\{20,50,200\}$. The three panels show the exact predictive density $\post_{i\mid i-1}(\cdot\mid \data{i-1})$ and the Wright--Fisher transition density $\trans{i}{i-1}(\cdot\mid \hat x_{i-1})$ started at the empirical composition.}
\label{fig:predictive-transition-harsh}
\end{figure}

We next illustrate the effect of the predictive approximation on a bounded
functional. In a more moderate two-type setting with $\aa=(2,2)$,
$x^{\*}=0.15$, and $h=0.05$, Figure~\ref{fig:diversity-twotype-main}
compares the distributions of Simpson diversity induced by the exact
predictive law and the empirical-start transition law. Unlike the law-level density
curves above, which are obtained by direct evaluation of the exact mixture
formulas, the curves in this figure are kernel density estimates based on
$100{,}000$ Monte Carlo draws from the two laws.
The agreement improves with the within-time sample size, providing a
distributional illustration of the effect of the plug-in approximation on
this functional. The corresponding comparison between the two predictive
laws themselves is reported in Section~\ref{SM:two-type-baseline} of the SM.

\begin{figure}[t!]
\centering
\includegraphics[width=\textwidth]{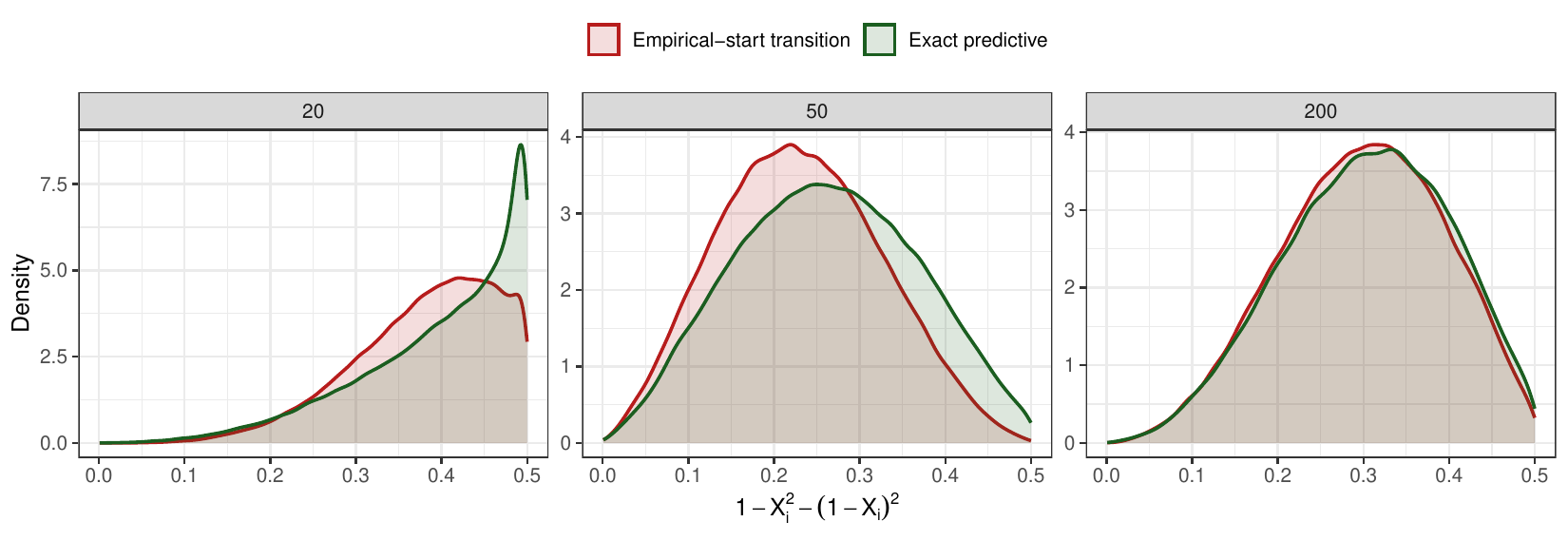}
\caption{\small Two-type predictive distribution of Simpson diversity with
$\aa=(2,2)$, $x^{\*}=0.15$, $h=0.05$, and
$n\in\{20,50,200\}$. From left to right, the panels correspond to
$n=20,50,200$ and compare the distributions of Simpson diversity induced by the exact
predictive law $\post_{i\mid i-1}(\cdot\mid \data{i-1})$ and the
empirical-start transition law
$\trans{i}{i-1}(\cdot\mid \hat x_{i-1})$. The displayed curves are kernel
density estimates based on $100{,}000$ Monte Carlo draws from each law.}
\label{fig:diversity-twotype-main}
\end{figure}

We next turn to local smoothing. At an unobserved intermediate time between two neighboring inspection times, the exact local smoother $\post_{i\mid i-1,i+1}(\cdot\mid \data{i-1},\data{i+1})$ is compared with the Wright--Fisher bridge $\bridge{i}{i-1}{i+1}(\cdot\mid \hat x_{i-1},\hat x_{i+1})$ obtained by plugging the empirical endpoint compositions into the bridge law. This is the comparison associated with Proposition~\ref{thm:bridge-main}. Figure~\ref{fig:local-smoother-bridge-main} shows that, in the same repeated-sampling regime, the exact local smoother is well captured by the empirical-endpoint bridge.

\begin{figure}[t!]
\centering
\includegraphics[width=\textwidth]{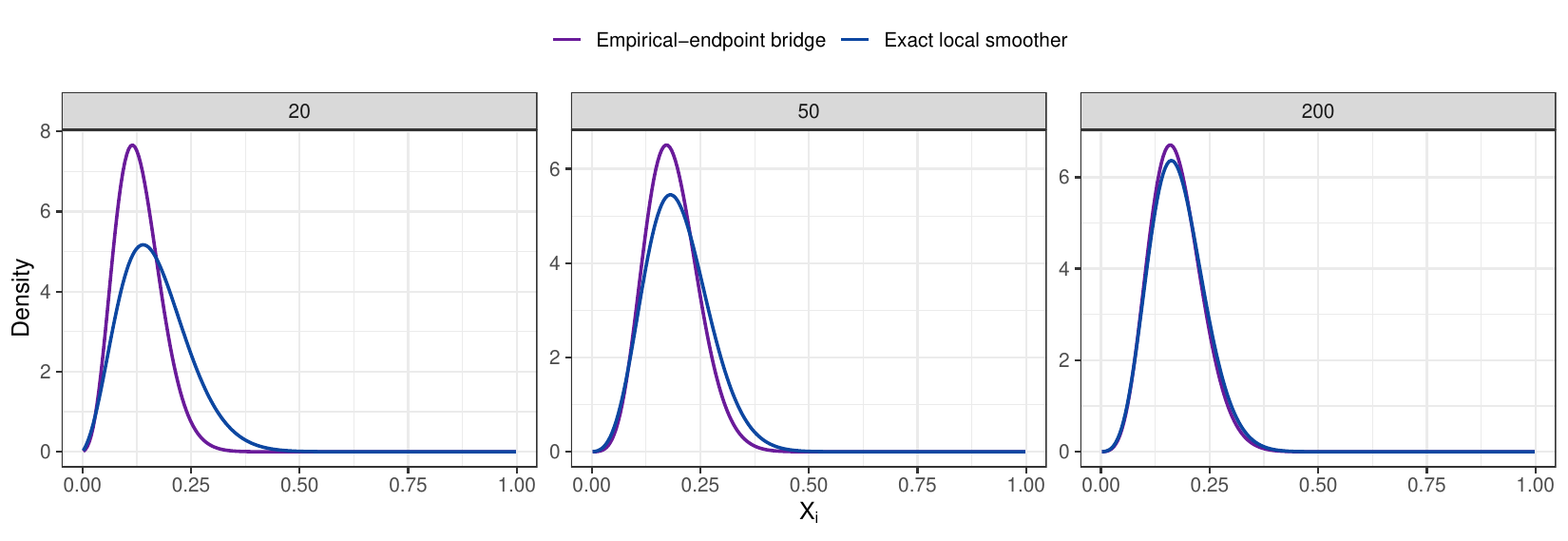}
\caption{\small Two-type local smoothing comparison for $\aa=(2,2)$ at the intermediate time $t_i=0.05$, with neighboring inspection times $t_{i-1}=0$ and $t_{i+1}=0.10$, so that $h_i=h_{i+1}=0.05$. The true latent path is $(x_{i-1}^*,x_i^*,x_{i+1}^*)=(0.15,0.20,0.17)$, the middle state is unobserved, and $n\in\{20,50,200\}$. Each panel compares the exact local smoother $\post_{i\mid i-1,i+1}(\cdot\mid \data{i-1},\data{i+1})$ with the Wright--Fisher bridge $\bridge{i}{i-1}{i+1}(\cdot\mid \hat x_{i-1},\hat x_{i+1})$ obtained from empirical endpoint compositions.}
\label{fig:local-smoother-bridge-main}
\end{figure}

Finally, we examine the sample-size dependence in
Corollary~\ref{thm:predictive-empirical}. We consider the true
compositions $x^*\in\{0.01,0.15,0.50\}$ at the representative gap
$h=0.10$. For each combination of $x^*$ and $n$, we generate $10$
independent multinomial samples conditional on the fixed true
composition $x^*$. Figure~\ref{fig:rate-behavior-main} reports the
largest value, across these datasets, of the total-variation quantity
on the left-hand side of Corollary~\ref{thm:predictive-empirical},
together with the $n$-dependent factor appearing in its bound. The
figure also reports, for each $(x^*,n)$ combination, the largest
absolute discrepancy between the predictive means of Simpson diversity
under the exact predictive law and the empirical-start transition law,
namely the quantity on the left-hand side of
\eqref{eq:simpson-bound}. Corresponding displays for the shorter and
longer gaps $h\in\{0.05,0.50\}$ are reported in
Section~\ref{SM:rate-gaps} of the SM. Both empirical summaries decrease
with $n$ on the same qualitative scale as the $n$-dependent factor in
Corollary~\ref{thm:predictive-empirical}.

\begin{figure}[t!]
\centering
\includegraphics[width=\textwidth]{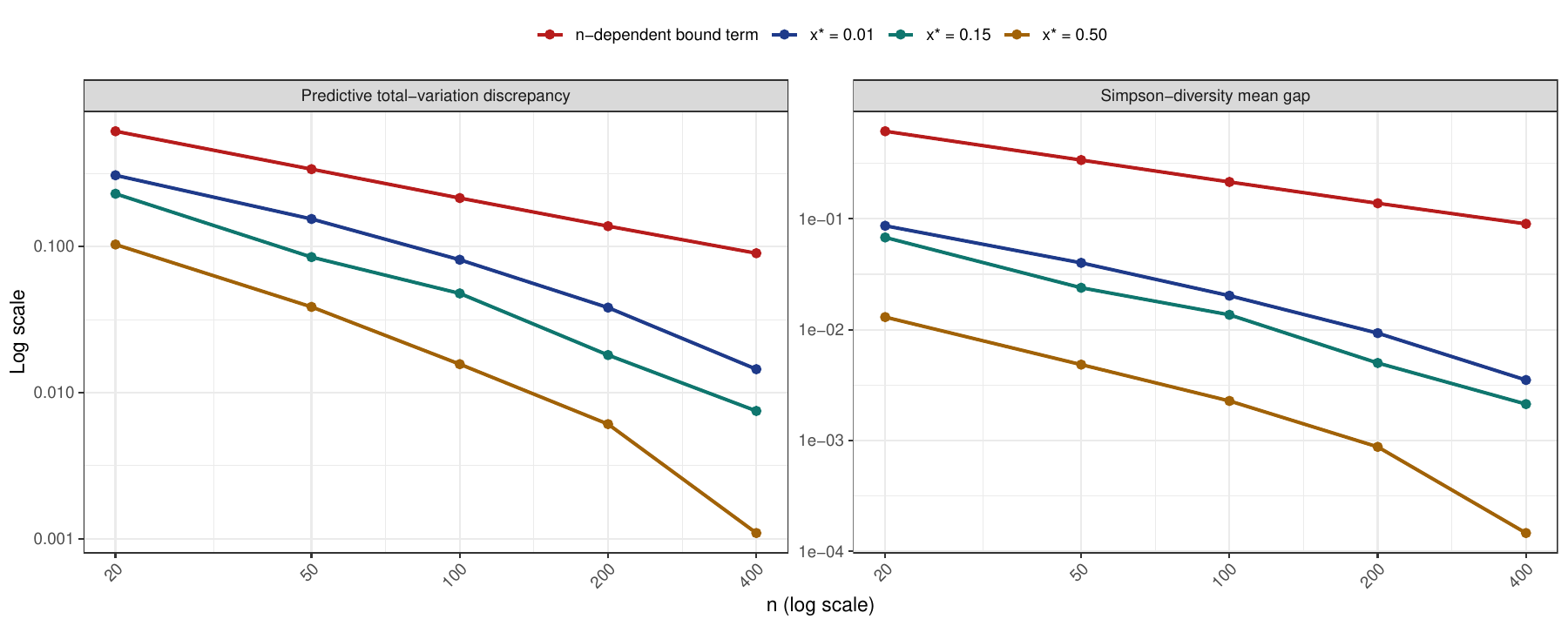}
\caption{\small Two-type predictive rate comparison with $\aa=(2,2)$,
$h=0.10$, $x^*\in\{0.01,0.15,0.50\}$, and
$n\in\{20,50,100,200,400\}$. The left panel shows the largest
total-variation discrepancy among the $10$ repeated datasets for each
$(x^*,n)$ combination, together with the $n$-dependent factor in
Corollary~\ref{thm:predictive-empirical}. The right panel shows the
corresponding largest absolute discrepancy in the predictive mean of
Simpson diversity, as defined by the left-hand side of
\eqref{eq:simpson-bound}, together with the same reference term.}
\label{fig:rate-behavior-main}
\end{figure}

Additional illustrations of the inspection-time Gaussian approximation and of a three-type boundary regime are given in Sections~\ref{SM:gaussian-inspection} and~\ref{SM:multitype} of the SM, while Section~\ref{SM:sensitivity} examines sensitivity to the mutation parameter and inspection gap.


\section{Concluding remarks}\label{sec:conclusion}

We studied exact prediction and smoothing in WF--HMMs under a repeated-sampling regime with fixed inspection times and growing within-time sample sizes. In this setting, the predictive and local smoothing laws converge to the corresponding WF transition and bridge laws, and the full joint posterior law at the inspection times concentrates at the target frequencies and satisfies a Bernstein--von Mises limit with distinct interior and boundary regimes. We also obtained explicit finite-sample control for the predictive law and a corresponding concentration bound for the marginal smoother.

These results clarify the role of temporal borrowing in the model. At finite sample size, information is propagated across inspection times through the coalescent-based smoothing recursions. In the present asymptotic regime, however, the dominant contribution comes from the multinomial sampling at each inspection time, while the induced longitudinal dependence does not affect the first-order inspection-time uncertainty.

Several questions remain open. The present regime is tailored to data sets with a fixed and relatively small number of inspection times and substantial within-time information, and does not address growing time grids, sparse counts, or more irregular observation schemes. Parameter learning is also outside the scope of the present paper. Our results are conditional on the WF parameters, but may still be used within data-augmentation schemes that alternate between latent-state and parameter updates. Conditional on the current parameter values, the transition- and bridge-based approximations apply to the latent-state update. A joint analysis of such procedures would require additional control over the parameter values explored by the algorithm. Developing a corresponding theory for learning mutation, selection, or other structural parameters would require combining the repeated-sampling regime studied here with an additional asymptotic framework.

A separate direction is to investigate whether the same asymptotic picture persists for related signal models, including coupled WF diffusions \citep{boetti_2024}, Poisson--Dirichlet hidden Markov models \citep{dallapria2026}, approximate filters based on discrete dual processes \citep{KKK24}, and nonparametric measure-valued signal models such as the Fleming--Viot settings studied in \citet{ascolani_predictive_2021,ascolani_smoothing_2023}. We leave these extensions for future work.

\section*{Acknowledgements}

The authors thank the Associate Editor and two anonymous reviewers for their careful reading and constructive comments, which helped substantially improve the paper.


\bibliographystyle{apalike}
\bibliography{bib}

@article{Maruotti2011,
  author  = {Maruotti, Antonello},
  title   = {Mixed Hidden {Markov} Models for Longitudinal Data: An Overview},
  journal = {International Statistical Review},
  year    = {2011},
  volume  = {79},
  number  = {3},
  pages   = {427--454},
  doi     = {10.1111/j.1751-5823.2011.00160.x}
}

@article{BansayeMeleardRichard2016,
  author  = {Bansaye, Vincent and M{\'e}l{\'e}ard, Sylvie and Richard, Mathieu},
  title   = {Speed of coming down from infinity for birth-and-death processes},
  journal = {Advances in Applied Probability},
  year    = {2016},
  volume  = {48},
  number  = {4},
  pages   = {1183--1210},
  doi     = {10.1017/apr.2016.70}
}

@article{sant2023ewf,
  title={{EWF: simulating exact paths of the Wright--Fisher diffusion}},
  author={Sant, Jaromir and Jenkins, Paul A and Koskela, Jere and Span{\`o}, Dario},
  journal={Bioinformatics},
  volume={39},
  number={1},
  pages={btad017},
  year={2023},
  publisher={Oxford University Press}
}

@article{Malaspinas2012,
  author={Malaspinas, Anna-Sapfo and Malaspinas, Orestis and Evans, Steven N and Slatkin, Montgomery},
  title={Estimating allele age and selection coefficient from time-serial data},
  journal={Genetics},
  volume={192},
  number={2},
  pages={599--607},
  year={2012},
  doi={10.1534/genetics.112.140939}
}

@article{Wutke2016,
  author={Wutke, Saskia and Benecke, Norbert and Sandoval-Castellanos, Edson and D{\"o}hle, Hans-J{\"u}rgen and Friederich, Susanne and Gonzalez, Javier and Hallsson, J{\'o}n Hallsteinn and Hofreiter, Michael and L{\~o}ugas, Lembi and Magnell, Ola and others},
  title={Spotted phenotypes in horses lost attractiveness in the {Middle Ages}},
  journal={Scientific Reports},
  volume={6},
  pages={38548},
  year={2016},
  doi={10.1038/srep38548}
}

@article{jenkins2017exact,
  title={{Exact simulation of the Wright-Fisher diffusion}},
  author={Jenkins, Paul and Span{\`o}, Dario},
  journal={Annals of Applied Probability},
  volume={27},
  number={3},
  pages={1478--1509},
  year={2017},
  publisher={Institute of Mathematical Statistics}
}

@article{DOnofrio2,
	title={{Jacobi processes with jumps as neuronal models: A first passage time analysis}},
	author={{D'Onofrio}, Giuseppe and Patie, Pierre and Sacerdote, Laura},
	journal={SIAM Journal on Applied Mathematics},
	volume={84},
	number={1},
	pages={189--214},
	year={2024},
	publisher={SIAM}
}

@article{Toscani,
	author = {Giuseppe Toscani},
	title = {{Kinetic models of opinion formation}},
	volume = {4},
	journal = {Communications in Mathematical Sciences},
	number = {3},
	publisher = {International Press of Boston},
	pages = {481 -- 496},
	year = {2006},
}

@article{Mena,
	title={Dynamic density estimation with diffusive {D}irichlet mixtures},
	author={Mena, Rams{\'e}s H. and Ruggiero, Matteo},
	journal={Bernoulli},
	volume={22},
	number={2},
	pages={901--926},
	year={2016}
}

@article{Perrone,
	title={Poisson random fields for dynamic feature models},
	author={Perrone, Valerio and Jenkins, Paul A and Spano, Dario and Teh, Yee Whye},
	journal={Journal of Machine Learning Research},
	volume={18},
	number={127},
	pages={1--45},
	year={2017}
}

@article{Filipovic,
	title={The {J}acobi stochastic volatility model},
	author={Ackerer, Damien and Filipovi{\'c}, Damir and Pulido, Sergio},
	journal={Finance and Stochastics},
	volume={22},
	pages={667--700},
	year={2018},
	publisher={Springer}
}

@article{Gourieroux,
	title={Multivariate {J}acobi process with application to smooth transitions},
	author={Gourieroux, Christian and Jasiak, Joann},
	journal={Journal of econometrics},
	volume={131},
	number={1-2},
	pages={475--505},
	year={2006},
	publisher={Elsevier}
}

@article{He2023,
	title={Estimating temporally variable selection intensity from ancient DNA data},
	author={He, Zhangyi and Wang, Zhenting and Beaumont, Mark and Yu, Feng},
	journal={Molecular Biology and Evolution},
	volume={40},
	number={3},
	pages={msad008},
	year={2023},
	doi={10.1093/molbev/msad008},
	publisher={Oxford University Press}
}

@article{Bouzid2021,
	title={Ancient {DNA} analysis},
	author={Bouzid, Nabil and Bravo, Guillermo A. and Austin, Jeremy J. and Cooper, Alan and Poinar, Hendrik N.},
	journal={Nature Reviews Methods Primers},
	volume={1},
	pages={14},
	year={2021},
	doi={10.1038/s43586-020-00011-0},
	publisher={Nature Publishing Group}
}

@article{Schraiber,
	title={{B}ayesian inference of natural selection from allele frequency time series},
	author={Schraiber, Joshua G and Evans, Steven N and Slatkin, Montgomery},
	journal={Genetics},
	volume={203},
	number={1},
	pages={493--511},
	year={2016},
	publisher={Oxford University Press}
}

@article{Steinrucken,
	title={A novel spectral method for inferring general diploid selection from time series genetic data},
	author={Steinr{\"u}cken, Matthias and Bhaskar, Anand and Song, Yun S},
	journal={Annals of Applied Statistics},
	volume={8},
	number={4},
	pages={2203},
	year={2014}
}

@article{Bollback,
	title={Estimation of {2Nes} from temporal allele frequency data},
	author={Bollback, Jonathan P and York, Thomas L and Nielsen, Rasmus},
	journal={Genetics},
	volume={179},
	number={1},
	pages={497--502},
	year={2008},
	publisher={Oxford University Press}
}

@article{Sant25,
	title={Bayesian inference from time series of allele frequency data using exact simulation techniques},
	author={Sant, Jaromir and Jenkins, Paul A and Koskela, Jere and Spano, Dario},
	journal={arXiv preprint arXiv:2502.12279},
	year={2025}
}

@article{douc2004,
  author  = {R{\'e}mi Douc and {\'E}ric Moulines and Tobias Ryd{\'e}n},
  title   = {Asymptotic properties of the maximum likelihood estimator in autoregressive models with {Markov} regime},
  journal = {Annals of Statistics},
  year    = {2004},
  volume  = {32},
  number  = {5},
  pages   = {2254}
}

@article{bickel1998,
  author  = {Peter J. Bickel and Ya'acov Ritov and Tobias Ryd{\'e}n},
  title   = {Asymptotic normality of the maximum-likelihood estimator for general hidden Markov models},
  journal = {Annals of Statistics},
  year    = {1998},
  volume  = {26},
  number  = {4},
  pages   = {1614--1635}
}

@book{cappe2005,
  author    = {Capp{\'e}, Olivier and Moulines, Eric and Ryd{\'e}n, Tobias},
  title     = {Inference in Hidden Markov Models},
  publisher = {Springer},
  year      = {2005}
}

@article{kalmanbucy1961,
  author  = {Kalman, R. E. and Bucy, R. S.},
  title   = {New Results in Linear Filtering and Prediction Theory},
  journal = {Journal of Basic Engineering},
  year    = {1961},
  volume  = {83},
  number  = {1},
  pages   = {95--108},
  doi     = {10.1115/1.3658902}
}

@article{kalman1960,
    Author = {Kalman, Rudolph Emil},
    Title = {{A New Approach to Linear Filtering and Prediction Problems}},
    Journal = {Transactions of the ASME--Journal of Basic Engineering},
    Volume = {82},
    Number = {Series D},
    Pages = {35--45},
    Year = {1960}
}

@article{dallapria2026,
  author  = {Dalla Pria, Marco and Ruggiero, Matteo and Span{\`o}, Dario},
  title   = {Exact inference via quasi-conjugacy in two-parameter {Poisson--Dirichlet} hidden {Markov} models},
  journal = {Journal of the American Statistical Association},
  year    = {2026},
note = {Advance online publication. DOI: \href{https://doi.org/10.1080/01621459.2026.2676717}{10.1080/01621459.2026.2676717}}
}

@article{ascolani_predictive_2021,
  author  = {Ascolani, Filippo and Lijoi, Antonio and Ruggiero, Matteo},
  title   = {Predictive inference with {Fleming--Viot}-driven dependent {Dirichlet} processes},
  journal = {Bayesian Analysis},
  year    = {2021},
  volume  = {16},
  number  = {2},
  pages   = {371--395}
}

@article{ascolani_smoothing_2023,
  author  = {Ascolani, Filippo and Lijoi, Antonio and Ruggiero, Matteo},
  title   = {Smoothing distributions for conditional {Fleming--Viot} and {Dawson--Watanabe} diffusions},
  journal = {Bernoulli},
  year    = {2023},
  volume  = {29},
  number  = {2},
  pages   = {1410--1434}
}

@article{boetti_2024,
  author  = {Boetti, Chiara and Ruggiero, Matteo},
  title   = {Filtering coupled {Wright--Fisher} diffusions},
  journal = {Journal of Mathematical Biology},
  year    = {2024},
  volume  = {89},
  pages   = {64}
}

@book{ethier_kurtz_1986,
  author    = {Ethier, Stewart N. and Kurtz, Thomas G.},
  title     = {Markov Processes: Characterization and Convergence},
  year      = {1986},
  publisher = {John Wiley \& Sons},
  address   = {New York},
  series    = {Wiley Series in Probability and Mathematical Statistics},
  isbn      = {9780471081861}
}

@article{griffiths1979transition,
  title={A transition density expansion for a multi-allele diffusion model},
  author={Griffiths, RC},
  journal={Advances in Applied Probability},
  volume={11},
  number={2},
  pages={310--325},
  year={1979},
  publisher={Cambridge University Press}
}

@article{griffiths_asymptotic_1984,
  author  = {Griffiths, Robert C.},
  title   = {Asymptotic line-of-descent distributions},
  journal = {Journal of Mathematical Biology},
  year    = {1984},
  volume  = {21},
  number  = {1},
  pages   = {67--75}
}

@article{kingman1982coalescent,
  title={The coalescent},
  author={Kingman, John Frank Charles},
  journal={Stochastic processes and their applications},
  volume={13},
  number={3},
  pages={235--248},
  year={1982},
  publisher={Elsevier}
}

@article{griffiths_wrightfisher_2018,
  author  = {Griffiths, Robert C. and Jenkins, Paul A. and Span{\`o}, Dario},
  title   = {Wright--Fisher diffusion bridges},
  journal = {Theoretical Population Biology},
  year    = {2018},
  volume  = {122},
  pages   = {67--77}
}

@article{KKK21,
  author  = {Kon Kam King, Guillaume and Papaspiliopoulos, Omiros and Ruggiero, Matteo},
  title   = {Exact inference for a class of hidden {Markov} models on general state spaces},
  journal = {Electronic Journal of Statistics},
  year    = {2021},
  volume  = {15},
  pages   = {2832--2875}
}

@article{KKK24,
  author  = {Kon Kam King, Guillaume and Pandolfi, Andrea and Piretto, Matteo and Ruggiero, Matteo},
  title   = {Approximate filtering via discrete dual processes},
  journal = {Stochastic Processes and their Applications},
  year    = {2024},
  volume  = {168},
  pages   = {104268}
}

@article{papaspiliopoulos_optimal_2014,
  author  = {Papaspiliopoulos, Omiros and Ruggiero, Matteo},
  title   = {Optimal filtering and the dual process},
  journal = {Bernoulli},
  year    = {2014},
  volume  = {20},
  number  = {4},
  pages   = {1999--2019}
}

@article{papaspiliopoulos_conjugacy_2016,
  author  = {Papaspiliopoulos, Omiros and Ruggiero, Matteo and Span{\`o}, Dario},
  title   = {Conjugacy properties of time-evolving {Dirichlet} and gamma random measures},
  journal = {Electronic Journal of Statistics},
  year    = {2016},
  volume  = {10},
  number  = {2},
  pages   = {3452--3489}
}

@article{tavare_1984,
  author  = {Tavar{\'e}, Simon},
  title   = {Line-of-descent and genealogical processes, and their applications in population genetics models},
  journal = {Theoretical Population Biology},
  year    = {1984},
  volume  = {26},
  number  = {2},
  pages   = {119--164}
}

@book{vaart_asymptotic_1998,
  author    = {van der Vaart, Aad W.},
  title     = {Asymptotic Statistics},
  year      = {1998},
  series    = {Cambridge Series in Statistical and Probabilistic Mathematics},
  publisher = {Cambridge University Press},
  address   = {Cambridge}
}

\clearpage
\appendix
\renewcommand{\thesection}{\Alph{section}}
\renewcommand{\thesubsection}{\thesection.\arabic{subsection}}
\renewcommand{\thetheorem}{SM.\thesection.\arabic{theorem}}
\renewcommand{\theaxiom}{SM.\arabic{axiom}}
\renewcommand{\theequation}{SM.\arabic{equation}}
\section*{Supplementary Material}
\addcontentsline{toc}{section}{Supplementary Material}

\section{Proofs and auxiliary results}
\label{sec:supplementary}


\subsection{Proof of Proposition \ref{prop:coal-tail}}\label{app:truncation}

\begin{proof}[Proof of Proposition~\ref{prop:coal-tail}]
Let $N_t$ denote the number of blocks of the $n$-coalescent at time $t$.
Similarly, let $T_M$ denote the first time at which exactly $M$ blocks remain. Then
\[
\sum_{k > M} p_{n, k}(t) = \prob{N_t > M \mid N_0 = n} = \prob{T_M > t},
\]
where
\[
T_M = \sum_{i = M+1}^n \tau_i, \quad \tau_i \stackrel{\mathrm{ind.}}{\sim} \mathrm{Exp}\!\left(\frac{i(i+\theta -1)}{2}\right).
\]
Let $s > 0$. By Markov's inequality,
\begin{align}
\sum_{k > M} p_{n, k}(t)
&=
\prob{ e^{\sum\nolimits_{i = M+1}^n s \tau_i} > e^{st}}
\leq e^{-st} \prod_{i = M+1}^n \mean{e^{ s \tau_i}}.
\end{align}
If $s < (M+1)(M+\theta )/2$, then
\[
\mean{e^{s\tau_i}} = \frac{1}{1- \frac{2s}{i(i+\theta -1)}} \quad \text{for every } i > M,
\]
which implies
\[
\sum_{k > M} p_{n, k}(t)
\leq e^{-st-\sum_{i = M+1}^n \log \left(1- \frac{2s}{i(i+\theta -1)}\right)}.
\]
Choose now $s = M^2/3$. Then
\[
\begin{aligned}
-\sum_{i = M+1}^n \log \left(1- \frac{2s}{i(i+\theta -1)}\right)
&\leq \frac{2}{3}\sum_{i = M+1}^n\frac{M^2}{i(i+\theta -1)}\left(1-\frac{2}{3}\frac{M^2}{i(i+\theta -1)} \right)^{-1}\\
& \leq rM^2\sum_{i = M+1}^n\frac{1}{i(i+\theta -1)},
\end{aligned}
\]
for some constant $r > 0$ independent of $M$. Bounding the sum by the corresponding integral gives
\[
rM^2\sum_{i = M+1}^n\frac{1}{i(i+\theta -1)}
\leq rM^2\int_{M}^\infty\frac{1}{x(x+\theta -1)}\, \de x \leq r'M,
\]
with another constant $r' > 0$. Therefore
\[
\sum_{k > M} p_{n, k}(t) \leq e^{-M^2t/3 + r'M}.
\]
The statement then follows by finding $c > 0$ and $g > 0$ such that $-M^2t/3+ r'M \leq -cM^2t+g$ for every $M$. This can be accomplished for example by choosing $c= 1/4$ and $g > \max_{y > 0} \, \{r'y-y^2t/12\} > 0$.

\end{proof}


\subsection{Proofs of Theorem~\ref{thm:predictive-rate} and Corollaries~\ref{thm:predictive-empirical}, \ref{cor:predictive-bounded-functional}, and~\ref{prop:predictive-target}}
\label{app:predictive-rate}
\label{app:predictive-rate}

We first need two preliminary lemmas.
\begin{lemma}\label{lem:transition-lipschitz}
For every $t>0$ there exists a finite constant
\[
L_t:=\sum_{k=0}^\infty k\,p_k(t)<\infty
\]
such that
\[
\TV{p_t(x,\cdot)}{p_t(y,\cdot)}\le L_t\|x-y\|_1,
\qquad x,y\in\Delta_d.
\]
\end{lemma}

\begin{proof}

First of all notice that $L_t$ is well-defined and finite. Indeed, Proposition~\ref{prop:coal-tail} implies
\[
p_k(t)\le R\exp\{-c(k-1)^2t\}
\]
for suitable constants $c,R>0$, and therefore
$\sum_k k\,p_k(t)<\infty$.

For $k\ge 0$ and $x\in\Delta_d$, let
\[
M_k^x(\kk):=\binom{k}{\kk}x^\kk,
\qquad \kk\in\mathbb{Z}_+^d,\quad |\kk|=k,
\]
be the probability mass function of a multinomial distribution with size $k$ and probability vector $x$. Then
\[
p_t(x,\cdot)=\sum_{k=0}^\infty p_k(t)\sum_{|\kk|=k} M_k^x(\kk)\,\pi_{\aa+\kk}(\cdot),
\]
and similarly for $p_t(y,\cdot)$. Let $A\subseteq\Delta_d$ be measurable. Since $0\le \pi_{\aa+\kk}(A)\le 1$, we have that
\[
\begin{aligned}
\bigl|p_t(x,A)-p_t(y,A)\bigr|
&=
\left|
\sum_{k=0}^\infty p_k(t)\sum_{|\kk|=k}
\bigl(M_k^x(\kk)-M_k^y(\kk)\bigr)\pi_{\aa+\kk}(A)
\right| \\
&\le
\sum_{k=0}^\infty p_k(t)\sum_{|\kk|=k}\bigl|M_k^x(\kk)-M_k^y(\kk)\bigr| \\
&=
2\sum_{k=0}^\infty p_k(t)\,\TV{M_k^x}{M_k^y},
\end{aligned}
\]
and therefore
\begin{equation}\label{eq:bound_TV_distance}
\TV{p_t(x,\cdot)}{p_t(y,\cdot)}
\le
2\sum_{k=0}^\infty p_k(t)\,\TV{M_k^x}{M_k^y}.
\end{equation}

Now let $U_1,\dots,U_k$ be i.i.d.\ with law $\Mult(x)$ and $V_1,\dots,V_k$ be i.i.d.\ with law $\Mult(y)$. If $T(u_1,\dots,u_k)$ denotes the vector of category counts, then
\[
T(U_1,\dots,U_k)\sim M_k^x,
\qquad
T(V_1,\dots,V_k)\sim M_k^y,
\]
which by the data processing inequality and the tensorization bound implies
\[
\TV{M_k^x}{M_k^y}
\le
\TV{\Mult(x)^{\otimes k}}{\Mult(y)^{\otimes k}} \leq k\,\TV{\Mult(x)}{\Mult(y)} = \frac{k}{2}\|x-y\|_1.
\]
Substituting this into \eqref{eq:bound_TV_distance} gives the result.
\end{proof}

\begin{lemma}\label{lem:dirichlet-mean-dev}
Let $X\sim\pi_{\beta}$ with parameter vector $\beta=(\beta_1,\dots,\beta_d)\in(0,\infty)^d$ and total mass $\beta_0=\sum_{j=1}^d\beta_j$. Let
\[
m=\mean{X}=\frac{\beta}{\beta_0}.
\]
Then
\[
\mean{\|X-m\|_1}\le \sqrt{\frac{d}{\beta_0+1}}.
\]
\end{lemma}

\begin{proof}
By Jensen's inequality we have that
\[
\mean{\|X-m\|_1}\le \sqrt{d\,\mean{\|X-m\|_2^2}}.
\]
Moreover for the Dirichlet distribution it holds that
\[
\mathrm{Var}(X_j)=\frac{m_j(1-m_j)}{\beta_0+1},
\qquad
j=1,\dots,d.
\]
Therefore
\[
\mean{\|X-m\|_2^2}
=
\sum_{j=1}^d \mathrm{Var}(X_j)
=
\frac{1-\sum_{j=1}^d m_j^2}{\beta_0+1}
\le
\frac{1}{\beta_0+1},
\]
and the claim follows.
\end{proof}

\begin{proof}[Proof of Theorem~\ref{thm:predictive-rate}]
By Bayes' formula at time $i-1$ and the Markov property of the latent signal,
\[
\post_{i\mid i-1}(\cdot\mid \data{i-1})
=
\int_{\Delta_d} p_{h_i}(x',\cdot)\,\pi_{\aa+\nn_{i-1}}(\de x').
\]
Therefore
\[
\begin{aligned}
\TV{\post_{i\mid i-1}(\cdot\mid \data{i-1})}{\trans{i}{i-1}(\cdot\mid x)}
&\le
\int_{\Delta_d}
\TV{p_{h_i}(x',\cdot)}{p_{h_i}(x,\cdot)}
\pi_{\aa+\nn_{i-1}}(\de x') \\
&\le
L_{h_i}\int_{\Delta_d}\|x'-x\|_1\,\pi_{\aa+\nn_{i-1}}(\de x'),
\end{aligned}
\]
by Lemma~\ref{lem:transition-lipschitz}. Let now
\[
m_{i-1}:=\frac{\aa+\nn_{i-1}}{\theta+n}
\]
be the mean of $\pi_{\aa+\nn_{i-1}}$. Then
\[
\int_{\Delta_d}\|x'-x\|_1\,\pi_{\aa+\nn_{i-1}}(\de x')
\le
\mean{\|X-m_{i-1}\|_1}+\|m_{i-1}-x\|_1,
\qquad X\sim\pi_{\aa+\nn_{i-1}}.
\]
By Lemma~\ref{lem:dirichlet-mean-dev},
\[
\mean{\|X-m_{i-1}\|_1}\le \sqrt{\frac{d}{n+\theta+1}}.
\]
Moreover,
\begin{equation}\label{eq:to_use_later}
\begin{aligned}
\|m_{i-1}-\hat x_{i-1}\|_1
=&\,
\sum_{j=1}^d
\left|
\frac{\alpha_j+n_{i-1,j}}{n+\theta}
-
\frac{n_{i-1,j}}{n}
\right|\\
=&\,
\sum_{j=1}^d
\frac{|n\alpha_j-\theta n_{i-1,j}|}{n(n+\theta)}
\le
\frac{n\theta+\theta\sum_{j=1}^d n_{i-1,j}}{n(n+\theta)}
=
\frac{2\theta}{n+\theta}.
\end{aligned}
\end{equation}
Hence
\[
\|m_{i-1}-x\|_1
\le
\|\hat x_{i-1}-x\|_1+\frac{2\theta}{n+\theta}.
\]
Combining the previous displays proves the first bound.

As regards the second part of the statement, note that
\[
\|\hat x_{i-1}-x_{i-1}^*\|_1>r
\quad\Longrightarrow\quad
\max_{1\le j\le d}|\hat x_{i-1,j}-x_{i-1,j}^*|>\frac{r}{d}.
\]
Thus, by a union bound,
\[
\mu_{i-1}^{(n)}\!\left(\|\hat x_{i-1}-x_{i-1}^*\|_1>r\right)
\le
\sum_{j=1}^d
\mu_{i-1}^{(n)}\!\left(|\hat x_{i-1,j}-x_{i-1,j}^*|>\frac{r}{d}\right).
\]
For each fixed $j$, the coordinate $\hat x_{i-1,j}$ is the empirical mean of $n$ i.i.d.\ Bernoulli variables with mean $x_{i-1,j}^*$. Hence Hoeffding's inequality gives
\[
\mu_{i-1}^{(n)}\!\left(|\hat x_{i-1,j}-x_{i-1,j}^*|>\frac{r}{d}\right)
\le
2\exp\!\left\{-2n\frac{r^2}{d^2}\right\},
\]
and summing over $j=1,\dots,d$ yields
\[
\mu_{i-1}^{(n)}\!\left(\|\hat x_{i-1}-x_{i-1}^*\|_1>r\right)
\le
2d\exp\!\left\{-2n\frac{r^2}{d^2}\right\}.
\]
The result now follows from the first part of the statement.
\end{proof}

\begin{proof}[Proofs of Corollaries~\ref{thm:predictive-empirical}, \ref{cor:predictive-bounded-functional}, and~\ref{prop:predictive-target}]
Corollary~\ref{thm:predictive-empirical} follows from the first part of Theorem~\ref{thm:predictive-rate} by taking $x=\hat x_{i-1}$. Corollary~\ref{cor:predictive-bounded-functional} then follows from
\[
\left|\int f\,\de P-\int f\,\de Q\right|
\le
2\|f\|_\infty\,d_{\mathrm{TV}}(P,Q).
\]
Finally, Corollary~\ref{prop:predictive-target} follows from Theorem~\ref{thm:predictive-rate} with $x=x_{i-1}^*$, since
$\|\hat x_{i-1}-x_{i-1}^*\|_1\to0$
$\mu_{i-1}^{(\infty)}$-almost surely and the remaining terms vanish as $n\to\infty$.
\end{proof}


\subsection{Proof of Proposition~\ref{thm:bridge-main}}\label{app:local-targets}
Denote
\begin{equation}\label{eq:wf-bridge-C-sec2}
C_{\kk, \kk'} =
\frac{
\Gamma(\theta + |\kk|)
\Gamma(\theta + |\kk'|)\prod_{j = 1}^d
\Gamma(\alpha_j +k_j + k'_j)
}{
\prod_{j = 1}^d\Gamma(\alpha_j +k_j)
\prod_{j = 1}^d\Gamma(\alpha_j +k'_j)
\Gamma(\theta + |\kk| + |\kk'|).
}
\end{equation}
Then recall from \eqref{eq:predictive-density-sec2} that the local smoother can be written as
\[
\post_{i\mid i-1,i+1}(x\mid \data{i-1},\data{i+1})
=\sum_{\kk\le\nn_{i-1},\,  \kk'\le\nn_{i+1}}
w^{\nn_{i-1},\nn_{i+1}}_{\kk,\kk'}(h_i,h_{i+1})\,
\pi_{\aa+\kk+\kk'}(x),
\]
with 
\[
w^{\nn_{i-1},\nn_{i+1}}_{\kk,\kk'}(h_i,h_{i+1})
\propto
C_{\kk,\kk'}\,
p_{\nn_{i-1},\kk}(h_i)\,
p_{\nn_{i+1},\kk'}(h_{i+1}).
\]
Similarly from \eqref{eq:wf-bridge-mixture-sec2} the law of the bridge is given by
\[
\bridge{i}{i-1}{i+1}(z\mid x,x')
=\sum_{\kk\in\Z^d}\sum_{\kk'\in\Z^d}
w_{\kk,\kk'}(h_i,h_{i+1})\,\pi_{\aa+\kk+\kk'}(z),
\]
where
\[
w_{\kk,\kk'}(h_i,h_{i+1})
\propto
C_{\kk,\kk'}\,p_{\kk}^{x}(h_i)\,p_{\kk'}^{x'}(h_{i+1}).
\]
We first need two technical results.
\begin{lemma}\label{lm:tech_convergence}
    Let $\kk \in \Z^d$ with $|\kk| = k$. Then we have that
    \[
    \binom{n}{k}^{-1}\binom{\nn}{\kk} \to \binom{k}{\kk}(x^*)^\kk
    \]
    $(x^*)^{\infty}$-a.s. as $n \to \infty.$
\end{lemma}
\begin{proof}
The statement is the well-known convergence of the multivariate hypergeometric distribution to the multinomial distribution, so we omit the proof.
\end{proof}
\begin{lemma}\label{lm:ineq_C}
Let $C_{\kk, \kk'}$ as in \eqref{eq:wf-bridge-C-sec2}. Then there exists $B:= B(\aa, d) > 0$ such that
\[
C_{\kk, \kk'} \leq Be^{k\log(k) + k'\log(k')},
\]
for every $\kk, \kk' \in \Z^d$ such that $|\kk| = k$ and $|\kk'| = k'$.
\end{lemma}
\begin{proof}
First of all, notice that
\[
\frac{\prod_{j = 1}^d\Gamma(\alpha_j +k_j + k'_j)}{\Gamma(\theta  + |\kk| + |\kk'|)} \leq \frac{\prod_{j = 1}^d\Gamma(\alpha_j)}{\Gamma(\theta)},
\]
and
\[
\prod_{j = 1}^d\Gamma(\alpha_j +k_j)
\prod_{j = 1}^d\Gamma(\alpha_j +k'_j) \geq c^{2d},
\]
where $c = \inf_{x > 0} \, \Gamma(x) > 0$. Moreover, by the Stirling's formula there exists $C = C(\theta) > 0$ such that
\[
\Gamma(\theta  + k) \leq Ck^{k}, \qquad k \geq 0,
\]
from which the result follows.
\end{proof}

We also need to recall the well-known fact that, for each fixed $t>0$ and $k\in\mathbb{N}$, $p_{k}(t)= p_{\infty, k}(t) :=\lim_{n\to\infty}p_{n,k}(t)$. We state this formally (without proof) for later reference.

\begin{lemma}
  \label{lemma: coalescent density convergence}
  For any \( t \in \real_+ \) and for any  $ k \in \integer_+ $, as $ n \rightarrow \infty$, it holds that $\ncoal{n}{k} \to \coal{k}$.
\end{lemma}

\begin{proof}[Proof of Proposition~\ref{thm:bridge-main}]

Let $\nn, \nn'$ be the vector of multiplicities associated to $\data{i-1}$ and $\data{i+1}$. First of all we rewrite $\post_{i | {i-1},{i+1}}( \cdot | \data{i-1},\data{i+1})$ and $\bridge{i}{{i-1}}{{i+1}}( \cdot | x^*_{{i-1}},x^*_{{i+1}})$ as
\begin{align*}
&\post_{i | {i-1},{i+1}}(z| \data{i-1},\data{i+1}) \\
&\quad= \frac{\sum_{k = 0}^n\sum_{k' = 0}^np_{n,k}(h_i)p_{n,k'}(h_{i+1}) \sum_{|\kk| = k}\binom{n}{k}^{-1}\binom{\nn}{\kk}\sum_{|\kk'| = k'}\binom{n}{k'}^{-1}\binom{\nn'}{\kk'}C_{\kk, \kk'}\pi_{\aa+\kk + \kk'}(z)}{\sum_{k = 0}^n\sum_{k' = 0}^np_{n,k}(h_i)p_{n,k'}(h_{i+1}) \sum_{|\kk| = k}\binom{n}{k}^{-1}\binom{\nn}{\kk}\sum_{|\kk'| = k'}\binom{n}{k'}^{-1}\binom{\nn'}{\kk'}C_{\kk, \kk'}}
\end{align*}
and
\begin{align*}
&\bridge{i}{{i-1}}{{i+1}}( z | x^*_{{i-1}},x^*_{{i+1}}) \\
&= \frac{\sum_{k = 0}^\infty \sum_{k' = 0}^\infty p_{k}(h_i) p_{k'}(h_{i+1})\sum_{|\kk| = k}\binom{k}{\kk}(x_{i-1}^*)^\kk\sum_{|\kk'| = k'}\binom{k'}{\kk'}(x_{i+1}^*)^{\kk'}C_{\kk, \kk'}\pi_{\aa+\kk + \kk'}(z)}{\sum_{k = 0}^\infty \sum_{k' = 0}^\infty p_{k}(h_i) p_{k'}(h_{i+1})\sum_{|\kk| = k}\binom{k}{\kk}(x_{i-1}^*)^\kk\sum_{|\kk'| = k'}\binom{k'}{\kk'}(x_{i+1}^*)^{\kk'}C_{\kk, \kk'}}.
\end{align*}
Let now $A \subset \Delta_d$ and define
\[
D_n(A) := \sum_{k = 0}^n\sum_{k' = 0}^np_{n,k}(h_i)p_{n,k'}(h_{i+1}) \sum_{|\kk| = k}\binom{n}{k}^{-1}\binom{\nn}{\kk}\sum_{|\kk'| = k'}\binom{n}{k'}^{-1}\binom{\nn'}{\kk'}C_{\kk, \kk'}\pi_{\aa+\kk + \kk'}(A)
\]
and
\[
D(A) = \sum_{k = 0}^\infty \sum_{k' = 0}^\infty p_{k}(h_i) p_{k'}(h_{i+1})\sum_{|\kk| = k}\binom{k}{\kk}(x_{i-1}^*)^\kk\sum_{|\kk'| = k'}\binom{k'}{\kk'}(x_{i+1}^*)^{\kk'}C_{\kk, \kk'}\pi_{\aa+\kk + \kk'}(A).
\]
Then in order to prove the result it suffices to show that $D_\nn(A) \to D(A)$ $\mu^{(\infty)}_{i-1,i+1}$-almost surely as $n \to \infty$ uniformly over $A$.

Fix $\epsilon > 0$. By Proposition \ref{prop:coal-tail} and Lemma \ref{lm:ineq_C} we have that
\begin{align*}
p_{n,k}(h_i)p_{n,k'}(h_{i+1}) \sum_{|\kk| = k}\binom{n}{k}^{-1}\binom{\nn}{\kk}\sum_{|\kk'| = k'}&\binom{n}{k'}^{-1}\binom{\nn'}{\kk'}C_{\kk, \kk'} \\
&\leq R^2Be^{k\log(k)- c(k-1)^2}e^{k'\log(k')- c(k'-1)^2}
\end{align*}
and
\begin{align*}
p_{k}(h_i)p_{k'}(h_{i+1}) \sum_{|\kk| = k}\binom{k}{\kk}(x_{i-1}^*)^\kk\sum_{|\kk'| = k'}&\binom{k'}{\kk'}(x_{i+1}^*)^{\kk'}C_{\kk, \kk'}\\
& \leq R^2Be^{k\log(k)- c(k-1)^2}e^{k'\log(k')- c(k'-1)^2},
\end{align*}
which imply that there exists $M > 0$ such that
\[
\sum_{k = M+1}^n\sum_{k' = 0}^np_{n,k}(h_i)p_{n,k'}(h_{i+1}) \sum_{|\kk| = k}\binom{n}{k}^{-1}\binom{\nn}{\kk}\sum_{|\kk'| = k'}\binom{n}{k'}^{-1}\binom{\nn'}{\kk'}C_{\kk, \kk'} \leq \epsilon
\]
and
\[
\sum_{k' = M+1}^n\sum_{k = 0}^np_{n,k}(h_i)p_{n,k'}(h_{i+1}) \sum_{|\kk| = k}\binom{n}{k}^{-1}\binom{\nn}{\kk}\sum_{|\kk'| = k'}\binom{n}{k'}^{-1}\binom{\nn'}{\kk'}C_{\kk, \kk'} \leq\epsilon,
\]
as well as
\[
\sum_{k > M}\sum_{k' = 0}^\infty p_{k}(h_i)p_{k'}(h_{i+1}) \sum_{|\kk| = k}\binom{k}{\kk}(x_{i-1}^*)^\kk\sum_{|\kk'| = k'}\binom{k'}{\kk'}(x_{i+1}^*)^{\kk'}C_{\kk, \kk'} \leq \epsilon
\]
and
\[
\sum_{k' > M}\sum_{k = 0}^\infty p_{k}(h_i)p_{k'}(h_{i+1}) \sum_{|\kk| = k}\binom{k}{\kk}(x_{i-1}^*)^\kk\sum_{|\kk'| = k'}\binom{k'}{\kk'}(x_{i+1}^*)^{\kk'}C_{\kk, \kk'} \leq \epsilon.
\]
Then in order to prove $D_\nn(A) \to D(A)$ $\mu^{(\infty)}_{i-1,i+1}$-almost surely as $n \to \infty$ uniformly over $A$, it suffices to show that
\[
p_{n,k}(h_i)p_{n,k'}(h_{i+1}) \binom{n}{k}^{-1}\binom{\nn}{\kk}\binom{n}{k'}^{-1}\binom{\nn'}{\kk'} \to p_{k}(h_i)p_{k'}(h_{i+1}) \binom{k}{\kk}(x_{i-1}^*)^\kk\binom{k'}{\kk'}(x_{i+1}^*)^{\kk'},
\]
$\mu^{(\infty)}_{i-1,i+1}$-almost surely as $n \to \infty$ for every $\kk, \kk' \in \Z^d$ such that $k = |\kk| \leq M$ and $k' = |\kk'| \leq M$. This follows by combining Lemmas \ref{lm:tech_convergence} and \ref{lemma: coalescent density convergence}.
\end{proof}


\subsection{Proof of Theorem \ref{cor:marginal-smoother-empirical}}\label{app:predictive-quant}

Fix $1<i<k$, and let $\tilde{\post}_{i \given i-1, i+1}\left( z \given x,x', \data{} \right)$ denote the density of $X_i$ conditional on $X_{i-1} = x$, $X_{i+1} = x'$, and data $\data{}$ at time $i$ with multiplicities $\nn \in \Z^d$. By \eqref{eq:wf-bridge-mixture-sec2} and Bayes's theorem,
\begin{equation}\label{eq:bridge_data}
\tilde{\post}_{i \given i-1, i+1}\left( z \given x,x', \data{} \right) = \sum_{\mathbf{k} \in \integer_+^d}
\sum_{\kk' \in \integer_+^d}
\tilde{w}_{\nn, \kk, \kk'}(h_i, h_{i+1})
\pi_{\aa+\nn+\mathbf{k} + \kk'}(z),
\end{equation}
where
\begin{equation}\label{def_w_tilde}
\tilde{w}_{\nn, \kk, \kk'}(h, h') \propto r_{\nn, \kk, \kk'}(h, h') := C_{\nn, \kk, \kk'}p_{\mathbf{k}}^{x}(h)p_{\kk'}^{x'}(h')
\end{equation}
and
\begin{equation}\label{def_C_n}
C_{\nn, \kk, \kk'} = \frac{\Gamma(\theta + |\kk|)}{\prod_{j = 1}^d\Gamma(\alpha_j +k_j)}\frac{\Gamma(\theta + |\kk'|)}{\prod_{j = 1}^d\Gamma(\alpha_j +k'_j)}\frac{\Gamma(\theta + |\nn|)}{\prod_{j = 1}^d\Gamma(\alpha_j +n_j)}\frac{\prod_{j = 1}^d\Gamma(\alpha_j +n_j +k_j + k'_j)}{\Gamma(\theta + |\nn| + |\kk| + |\kk'|)}.
\end{equation}
We analogously define $\tilde{\post}_{i \given i-1}\left( z \given x, \data{} \right)$ and $\tilde{\post}_{i \given i+1}\left( z \given x', \data{} \right)$ when conditioning only on the past or future latent state.

We need several technical lemmas.
\begin{lemma}\label{lm:ineq_C_n}
Let $C_{\nn, \kk, \kk'}$ as in \eqref{def_C_n}. Then there exists $B = B(\theta, d) > 0$ such that
\[
C_{\nn, \kk, \kk'} \leq Be^{k\log(k) + k'\log(k')},
\]
for every $\nn$, $\kk, \kk' \in \Z^d$ such that $|\kk| = k$ and $|\kk'| = k'$.
\end{lemma}
\begin{proof}
Notice that
\[
\frac{\Gamma(\theta  + |\nn|)}{\prod_{j = 1}^d\Gamma(\alpha_j +n_j)}\frac{\prod_{j = 1}^d\Gamma(\alpha_j +n_j +k_j + k'_j)}{\Gamma(\theta  + |\nn| + |\kk| + |\kk'|)} = \frac{\int_{\Delta_d} x^{\nn + \kk + \kk'} \pi_{\alpha}(x)\, \text{d}x}{\int_{\Delta_d} x^{\nn} \pi_{\alpha}(x)\, \text{d}x} \leq 1,
\]
so that the result follows by reasoning as in Lemma \ref{lm:ineq_C}.
\end{proof}
\begin{lemma}\label{lem:smoother-remainder}
Let $\tilde{w}_{\nn, \kk, \kk'}$ be as in \eqref{def_w_tilde}. Then for every $h$ and $h'$ positive constants there exist $B_{h, h'} > 0$ and $C_{h, h'} > 0$ such that for every $M > 1$ we have that
\begin{equation}\label{ineq_w_tilde}
\sum_{|\kk| > M}
\sum_{|\kk'| \geq 0}
\tilde{w}_{\nn, \kk, \kk'}(h, h')
+ \sum_{|\kk'| > M}
\sum_{|\kk| \geq 0}
\tilde{w}_{\nn, \kk, \kk'}(h, h') < C_{h, h'}e^{-B_{h, h'}M^2},
\end{equation}
for every $\nn \in \Z^d$ and $(x, x') \in \Delta_d^2$.
\end{lemma}

\begin{proof}
By \eqref{def_w_tilde} we have that, for $r_{\nn, \kk, \kk'}$ as in \eqref{def_w_tilde},
\begin{equation}\label{eq:lower_normalizing_const}
\sum_{\mathbf{k} \in \integer_+^d}
\sum_{\kk' \in \integer_+^d}
r_{\nn, \kk, \kk'}(h, h') \geq p_{0}(h)p_{0}(h')\left(\frac{\Gamma(\theta )}{\prod_{j = 1}^d\Gamma(\alpha_j)}\right)^2 =: r > 0,
\end{equation}
which implies that the normalizing constant of $\{r_{\nn, \kk, \kk'}(h, h') \}_{\kk, \kk'}$ is uniformly bounded away from zero.

Moreover, by Proposition~\ref{prop:coal-tail} and Lemma~\ref{lm:ineq_C_n} we have that
\[
C_{\nn, \kk, \kk'}p_{\mathbf{k}}(h)p_{\kk'}(h') \leq R^2Be^{-c(k-1)^2h+k\log(k)}e^{-c(k'-1)^2h'+k'\log(k')} \leq Ke^{-sk^2- s(k')^2},
\]
for suitable constants $K>0$ and $s > 0$ depending on $h$ and $h'$. Since
\[
\sum_{k = 0}^\infty e^{-sk^2} < \infty,
\]
we have that
\[
\begin{aligned}
\sum_{|\kk| > M}
\sum_{|\kk'| \geq 0}
\tilde{w}_{\nn, \kk, \kk'}(h, h') &\leq \frac{K}{r}\left(\sum_{k' = 0}^\infty e^{-s(k')^2}\right)\sum_{k > M}e^{-sk^2}\\
& \leq \frac{K}{r}\left(\sum_{k' = 0}^\infty e^{-s(k')^2}\right)\left(\sum_{k > M}e^{-sk^2+sM^2}\right)e^{-sM^2},
\end{aligned}
\]
so that the result holds with $B_{h, h'} = s$ and 
\[
C_{h, h'} = \frac{K}{r}\left(\sum_{k' = 0}^\infty e^{-s(k')^2}\right)^2.
\]
\end{proof}
\begin{lemma}\label{lem:dirichlet-concentration}
Let $\beta\in(0,\infty)^d$ and $Z\sim\pi_\beta$. Let $m=\beta/\beta_0$, where $\beta_0=\sum_j\beta_j$. If $\|m-x\|_2\le \delta/2$, then
\[
\prob{\|Z-x\|_2\ge \delta}\le \frac{4}{(\beta_0+1)\delta^2}.
\]
\end{lemma}

\begin{proof}
If $\|m-x\|_2\le \delta/2$, then
\[
\{\|Z-x\|_2\ge \delta\}\subseteq \{\|Z-m\|_2\ge \delta/2\}.
\]
Hence, by Markov's inequality,
\[
\prob{\|Z-x\|_2\ge \delta}
\le
\prob{\|Z-m\|_2\ge \delta/2}
\le
\frac{4}{\delta^2}\sum_{j=1}^d \mathrm{Var}(Z_j).
\]
Now
\[
\mean{\|Z-m\|_2^2}
=
\sum_{j=1}^d \mean{(Z_j-m_j)^2}
=
\sum_{j=1}^d \mathrm{Var}(Z_j).
\]
The result follows from the fact that $\sum_{j=1}^d\mathrm{Var}(Z_j)\leq(\beta_0+1)^{-1}$.
\end{proof}

We now state the main result of this section in the next corollary.
\begin{corollary}\label{cor:local-conditional-empirical}
Fix $r>0$. Then there exists a constant $A_{h_i, h_{i+1}}$ such that
\[
\sup_{(x,x')\in\Delta_d^2}
\int_{\{z\in\Delta_d:\|z-\hat x_i\|_2\ge r/\sqrt{n}\}}
\tilde{\post}_{i\given i-1,i+1}(z\given x,x',\data{i})\,\de z
\le
\frac{A_{h_i, h_{i+1}}}{n}+\frac{16}{r^2},
\]
for every $n > N$, with $N := N(\aa, r, h_i, h_{i+1})$.
\end{corollary}

\begin{proof}
Fix $(x,x')\in\Delta_d^2$. By \eqref{eq:bridge_data} and Lemma \ref{lem:smoother-remainder} we have that
\[
\begin{aligned}
&\int_{\{z:\|z-\hat x_i\|_2\ge r/\sqrt{n}\}}
\tilde{\post}_{i\given i-1,i+1}(z\given x,x',\data{i})\,\de z\\
&\, \le
\sum_{|\kk|\le M}\sum_{|\kk'|\le M}\tilde{w}_{\nn, \kk, \kk'}(h_i, h_{i+1})
\int_{\{z:\|z-\hat x_i\|_2\ge r/\sqrt{n}\}}
\pi_{\aa+\nn+\kk+\kk'}(z)\,\de z
+ C_{h, h'}e^{-B_{h, h'}M^2},
\end{aligned}
\]
for every $M> 1$.
For fixed $\kk,\kk'$ with $|\kk|,|\kk'|\le M$, let
\[
\beta=\aa+\nn+\kk+\kk',\qquad
\beta_0=n+\theta+|\kk|+|\kk'|,\qquad
m=\frac{\beta}{\beta_0}.
\]
Reasoning as in \eqref{eq:to_use_later} we obtain that
\[
\|m-\hat x_i\|_1\le \frac{2(\theta+|\kk|+|\kk'|)}{n+\theta+|\kk|+|\kk'|}
\le \frac{2(\theta+2M)}{n+\theta}.
\]
Choose now $M = \lceil \sqrt{\log n}/\sqrt{B_{h, h'}}  +1 \rceil $. Therefore by the above calculations we get that
\[
\|m-\hat x_i\|_2\le\frac{r}{2\sqrt{n}},
\]
for $n$ bigger than some $N$ depending only on $\aa$, $r$, $h_i$ and $h_{i+1}$ (through $B_{h_i,h_{i+1}}$). Then
\[
\left\{z:\|z-\hat x_i\|_2\ge \frac{r}{\sqrt{n}} \right\}\subseteq \left\{z:\|z-m\|_2\ge\frac{r}{2\sqrt{n}}\right\}.
\]
Applying Lemma~\ref{lem:dirichlet-concentration} with $x$ replaced by $\hat x_i$ gives
\[
\int_{\{z:\|z-\hat x_i\|_2\ge r/\sqrt{n}\}}
\pi_{\aa+\nn+\kk+\kk'}(z)\,\de z
\le \frac{16n}{(n+\theta+1)r^2}.
\]
The result then follows with $A_{h,h'} = C_{h, h'}$.
\end{proof}

\begin{proof}[Proof of Theorem~\ref{cor:marginal-smoother-empirical}]
By the HMM structure and the Markov property,
\[
\post_{i\mid 1:k}(z\mid \data{1:k})
=
\int_{\Delta_d}\int_{\Delta_d}
\tilde{\post}_{i\given i-1,i+1}(z\given x,x',\data{i})\,
\post_{i-1,i+1\mid 1:k}(x,x'\mid \data{1:k})\,\de x\,\de x',
\]
where $\post_{i-1, i+1 | {1:k}}( x, x' | \data{1:k})$ is the joint density of the signal at times $t_{i-1}$ and $t_{i+1}$ conditional on $\data{1:k}$. Therefore, for the set
\[
B_r:= \left\{z\in\Delta_d:\|z-\hat x_i\|_2\ge \frac{r}{\sqrt{n}}\right\},
\]
we have
\[
\begin{aligned}
\post_{i\mid 1:k}(B_r\mid \data{1:k})
&=
\int_{\Delta_d}\int_{\Delta_d}
\left\{\int_{B_r}
\tilde{\post}_{i\given i-1,i+1}(z\given x,x',\data{i})\,\de z\right\}
\post_{i-1,i+1\mid 1:k}(x,x'\mid \data{1:k})\,\de x\,\de x'\\
&\le
\sup_{(x,x')\in\Delta_d^2}
\int_{B_r}
\tilde{\post}_{i\given i-1,i+1}(z\given x,x',\data{i})\,\de z.
\end{aligned}
\]
The result then follows by Corollary~\ref{cor:local-conditional-empirical}. 
\end{proof}


\subsection{Proof of Theorem~\ref{thm:consistency-bvm}}\label{app:bvm-aux}

For $\delta>0$ and $x^*\in\Delta_d$, let
\begin{equation}\label{def:set_A}
A_{\delta, x^*} := \left\{ x \in \Delta_d :  \|x-x^*\|_2 \geq \delta\right\},
\end{equation}
where $\|\cdot\|_2$ denotes the Euclidean norm on $\mathbb{R}^d$. We first need a technical result.
\begin{proposition}\label{prop:convergence}
Let $\delta > 0$. Then
\begin{equation}\label{eq:first_conv}
\sup_{(x, x') \in \Delta_d^2} \, \int_{A_{\delta, x_i^*}}\tilde{\post}_{i \given i-1, i+1}\left( z \given x,x', \data{} \right)\,\de z \to 0,
\end{equation}
as well as
\[
\sup_{x \in \Delta_d} \, \int_{A_{\delta, x_i^*}}\tilde{\post}_{i \given i-1}\left( z \given x', \data{} \right)\,\de z \to 0,
\]
and
\[
\sup_{x' \in \Delta_d} \, \int_{A_{\delta, x_i^*}}\tilde{\post}_{i \given i+1}\left( z \given x', \data{} \right)\,\de z \to 0
\]
$\mu_i^\infty$-a.s.\ as $n \to \infty$.
\end{proposition}

\begin{proof}
We prove \eqref{eq:first_conv}; the other two limits are analogous. Fix $\epsilon > 0$. By Lemma~\ref{lem:smoother-remainder} there exists $M > 0$, independent of $x$ and $x'$, such that
\[
\begin{aligned}
\sup_{(x, x') \in \Delta_d^2} \, &\int_{A_{\delta, x_i^*}}\tilde{\post}_{i \given i-1, i+1}( z \given x,x',  \data{} ) \, \de z \\
\leq&\, \sum_{|\kk| \leq M}
   \sum_{|\kk'| \leq  M}\int_{A_{\delta, x_i^*}}\pi_{\aa+\nn+\mathbf{k} + \kk'}(z)\,\de z + \frac{\epsilon}{2}.
\end{aligned}
\]
It is clear that for every $(\kk, \kk')$ such that $|\kk| \leq M$ and $|\kk'| \leq M$ we have that
\[
\int_{A_{\delta, x_i^*}}\pi_{\aa+\nn+\kk+\kk'}(z)\,\de z \to 0
\]
$\mu_i^\infty$-almost surely as $n \to \infty$. The result follows by arbitrariness of $\epsilon$.
\end{proof}

We now prove the first part of Theorem \ref{thm:consistency-bvm}.
\begin{proof}[Proof of \eqref{joint-consistency} in Theorem \ref{thm:consistency-bvm}]
Let $(Z^{(n)}_1, \dots, Z^{(n)}_k)$ be a random vector with probability density $\post_{{1:k} | {1:k}}( \cdot | \data{1:k})$. Then we prove the stronger statement that $Z_i^{(n)} \to x_i^*$ in probability $\mu^{(\infty)}_{1:k}$-a.s. as $n \to \infty$, for every $i = 1, \dots, k$.

Fix $1 < i <k$, $\delta > 0$ and let $\post_{i | {1:k}}( \cdot | \data{1:k})$ be the density of $Z_i^{(n)}$. Then we have to show that
\[
\int_{A_{\delta, x_i^*}}\post_{i | {1:k}}( z | \data{1:k}) \, \de z \to 0,
\]
$\mu^{(\infty)}_{1:k}$-a.s. as $n \to \infty$. By the Markov property and the HMM structure, we can write
\[
\post_{i | {1:k}}( z | \data{1:k}) = \int_{\Delta_d}\int_{\Delta_d}\tilde{\post}_{i \given i-1, i+1}( z \given x,x', \data{i} )\post_{i-1, i+1 | {1:k}}( x, x' | \data{1:k}) \, \de x \de x',
\]
with $\tilde{\post}_{i \given i-1, i+1}( z \given x,x', \data{i} )$ as in \eqref{eq:bridge_data} and where $\post_{i-1, i+1 | {1:k}}( x, x' | \data{1:k})$ is the joint density of the signal at times $t_{i-1}$ and $t_{i+1}$ conditional on $\data{1:k}$. By Proposition \ref{prop:convergence} we have that
\[
\int_{A_{\delta, x_i^*}}\post_{i | {1:k}}( z | \data{1:k}) \, \de z \leq \sup_{(x, x') \in \Delta_d^2} \, \int_{A_{\delta, x_i^*}}\tilde{\post}_{i \given i-1, i+1}( z \given x,x', \data{i} ) \, \de z \to 0,
\]
$\mu^{\infty}_{i}$-a.s. as $n \to \infty$.
If now $i = 1$ or $i = k$, the result follows analogously by replacing $\tilde{\post}_{i \given i-1, i+1}( z \given x,x', \data{} )$ with $\tilde{\post}_{i \given i+1}( z \given x', \data{i} )$ or $\tilde{\post}_{i \given i-1}( z \given x, \data{i} )$ respectively.
\end{proof}

The proof of the second part of Theorem~\ref{thm:consistency-bvm} hinges on the asymptotic behaviour of the Dirichlet components appearing in the local conditional mixture representations. Under an interior condition, this is provided by the standard Bernstein--von Mises theorem for the Dirichlet--Multinomial model. The next lemma extends that approximation to boundary configurations. It shows that positive-frequency coordinates retain the usual Gaussian $n^{-1/2}$-fluctuations, while coordinates with zero true frequencies fluctuate on the smaller $n^{-1}$ scale and converge to Gamma limits. These two components are asymptotically independent.

\begin{lemma}\label{lem:boundary-dirichlet}
For some \(1\le r\le d\), let
$
x^*=(x_1^*,\ldots,x_r^*,0,\ldots,0)\in\Delta_d,
$
with \(x_j^*>0\) for \(j=1,\ldots,r\). Let \(\nn=(n_1,\ldots,n_r,0,\dots,0)\sim\mathrm{MN}(n,x^*)\). Fix \(\eta\in\Z^d\) and consider the (conditionally on $\nn$) Dirichlet-distributed random vector
$
X^{(n)}\sim \pi_{\aa+\nn+\eta}.
$
Write \(X^{(n)}=(X^{(n),+},X^{(n),0})\), with `active' and `inactive' blocks defined by the first \(r\) and last \(d-r\) coordinates of \(X^{(n)}\), and write
$
\hat x=(n_1/n,\ldots,n_d/n).
$
Define
$
\Sigma^+=\operatorname{diag}(x^{*+})-x^{*+}x^{*+\top}.
$
Then, as \(n\to\infty\),
\[
\mathcal L\big(\sqrt n(X^{(n),+}-\hat x^+)\big)
\Rightarrow
\mathcal N(\cdot;0,\Sigma^+)
\]
in probability, and
\[
\mathcal L(nX^{(n),0})
\Rightarrow
\bigotimes_{j=r+1}^{d}
\Gamma(\alpha_j+\eta_j,1),
\]
in probability. If \(r=d\), the above product is interpreted as a point mass at the empty vector. In this case, the Gaussian limit also holds in total variation. The above convergences hold jointly, with limiting law equal to the product of the individual limiting laws.
\end{lemma}

\begin{proof}
We use the gamma representation of the Dirichlet distribution. In particular, we have
$$
	X_j^{(n)}\overset{d}{=}\frac{G_j^{(n)}}{G_0^{(n)}},
	\qquad j=1,\dots,d,
$$
for independent gamma random variables
$
G_j^{(n)}\sim \Gamma(\alpha_j+n_j+\eta_j,1),\ j=1,\ldots,d,
$
and where $
G_0^{(n)}=\sum_{\ell=1}^d G_\ell^{(n)}\sim
\Gamma\left(n+\sum_{j=1}^d\alpha_j+|\eta|,1\right)$. It holds that
\(G_0^{(n)}/n\to1\) in probability, and since, for \(j>r\), we have \(n_j=0\), it follows
\[
nX^{(n),0}
=
\frac{n}{G_0^{(n)}}(G_{r+1}^{(n)},\dots,G_d^{(n)})
\Rightarrow
(G_{r+1},\dots,G_d),
\]
by Slutsky's theorem, where $G_j\sim\Gamma(\alpha_j+\eta_j,1)$ are independent. This proves the gamma limit for the inactive block.

We now consider the active coordinates. Define
\[
H_j^{(n)}
=
\frac{G_j^{(n)}-n\hat x_j}{\sqrt n},
\qquad j=1,\ldots,r.
\]
Since \(n\hat x_j=n_j\), \(n_j/n\to x_j^*\), and \(\alpha_j+\eta_j\) is fixed, we obtain from the central limit theorem that
$
H^{(n)}=(H_1^{(n)},\ldots,H_r^{(n)})
\Rightarrow
H\sim N\bigl(0,\operatorname{diag}(x^{*+})\bigr)
$
in probability. Next, we note
\[
\sqrt n(X_j^{(n)}-\hat x_j)
=
\frac{n}{G_0^{(n)}}
\left\{
H_j^{(n)}
-
\hat x_j
\frac{G_0^{(n)}-n}{\sqrt n}
\right\}.
\]
with
\[
\frac{G_0^{(n)}-n}{\sqrt n}
=
\sum_{\ell=1}^r H_\ell^{(n)}
+
\frac{\sum_{\ell=r+1}^d G_\ell^{(n)}}{\sqrt n}.
\]
The second term on the right-hand side vanishes in probability since  $G_\ell^{(n)}\sim \Gamma(\alpha_\ell + \eta_\ell,1)$ for $\ell=r+1,\dots,d$. Hence, since, as observed above, \(n/G_0^{(n)}\to1, \hat x_j\to x^*_j\)in probability,
\[
\sqrt n(X_j^{(n)}-\hat x_j)
=
H_j^{(n)}
-
x^*_j\sum_{\ell=1}^r H_\ell^{(n)}
+
o_{\mathbb P}(1),
\qquad j=1,\ldots,r,
\]
and by the continuous mapping theorem,
$
\sqrt n(X^{(n),+}-\hat x^+)
\Rightarrow
N(\cdot;0,\Sigma^+),
$
where
$
\Sigma^+
=
\operatorname{diag}(x^{*+})
-
x^{*+}x^{*+\top}.
$
The limiting covariance is obtained through the linear map
\(H\mapsto H-x^{*+}(1^\top H)\) applied to
\(H\sim N(0,\operatorname{diag}(x^{*+}))\).

For the joint limit, note that by independence and the above considerations,
$$
(H_1^{(n)},\ldots,H_r^{(n)},G_{r+1}^{(n)},\ldots,G_d^{(n)})
\Rightarrow N(0,\operatorname{diag}(x^{*+})) \otimes
\bigotimes_{j=r+1}^{d}
\Gamma(\alpha_j+\eta_j,1).
$$
Moreover, as shown above,
\[
\sqrt n(X^{(n),+}-\hat x^+)
=
H^{(n)}
-
x^{*+}\,1^\top H^{(n)}
+
o_{\mathbb P}(1),
\qquad
nX^{(n),0}
=
(G_{r+1}^{(n)},\ldots,G_d^{(n)})
+
o_{\mathbb P}(1).
\]
Another application of Slutsky's theorem and the continuous mapping theorem then proves the claim.

Finally, if \(r=d\), there is no inactive block. In that case, the usual finite-dimensional Bernstein--von Mises theorem for regular statistical models, e.g., Theorem $10.1$ in \cite{vaart_asymptotic_1998}, gives the Gaussian convergence in total variation.
\end{proof}

\begin{lemma}\label{lm:decomposition_TV}
Let $\{\mu_n\}_n$ and $\{\nu_n\}_n$ be two sequences of probability measures on $\sX = \sX_1 \times \dots \times \sX_k$, where each $\sX_i$ is a subset of a Euclidean space and $\sX$ carries the Euclidean metric, which can be decomposed as
\[
\mu_n(\de x_{1:k}) = \mu_{n, 1}(\de x_1)\prod_{i = 2}^k\mu_{n, i}(\de x_i | x_{1:i-1}), \quad \nu_n(\de x_{1:k}) = \prod_{i = 1}^k\nu_{n, i}(\de x_i),
\]
where $\mu_{n, 1}(\de x_1)$ is a probability measure on $\sX_1$ and $\nu_{n, i}(\de x_i)$, $i=1,\dots,k$, are probability measures on $\sX_i$, while $\mu_{n, i}(\de x_i | x_{1:i-1})$ are transition probability kernels on $\sX_i$. Let $d(\cdot,\cdot)$ be either the total variation, $d_{\mathrm{TV}}$, or the bounded-Lipschitz, $d_{\mathrm{BL}}$, distance between probability measures. Assume that
\[
d(\mu_{n, 1}(\de x_1),\nu_{n, 1}(\de x_1)) \to 0, \quad \sup_{x_{1:i-1}} \, d(\mu_{n, i}(\de x_i | x_{1:i-1}),\nu_{n, i}(\de x_i)) \to 0,
\]
as $n \to \infty$ for every $i = 2, \dots, k$. Then $d(\mu_n,\nu_n) \to 0$, as $n \to \infty$.
\end{lemma}

\begin{proof}
We prove the result for $k = 2$, as the general case follows by induction, applying the case $k=2$ to $\sX_1\times\dots\times\sX_{k-1}$ and $\sX_k$, with $\nu_{n,1}\otimes\dots\otimes\nu_{n,k-1}$ in place of $\nu_{n,1}$.
Let $\{\rho_n\}_n$ be a sequence of probability measures on $\sX$ defined as
\[
\rho_n(\de x) = \mu_{n, 1}(\de x_1)\nu_{n, 2}(\de x_2).
\]
By the triangular inequality we have that
\begin{equation}\label{eq:triangle_TV}
d(\mu_n,\nu_n) \leq d(\mu_n,\rho_n)+d(\rho_n,\nu_n).
\end{equation}
If $d = d_{\mathrm{TV}}$, let $\mathcal{F}(A)$ be the class of all measurable functions $f \, : \, A \, \to \, [0,1]$, while if $d = d_{\mathrm{BL}}$, let it denote the class of functions $f$ with $\|f\|_\infty\le 1$ and Lipschitz constant smaller than $1$. Then we can write
\[
\begin{aligned}
&d(\rho_n,\nu_n)\\
&\quad = \sup_{f \in \mathcal{F}(\sX)} \left\lvert\int_\sX f(x) \rho_n(\de x) - \int_\sX f(x) \nu_n(\de x)\right\rvert\\
&\quad = \sup_{f \in \mathcal{F}(\sX)} \left\lvert\int_{\sX_1} \int_{\sX_2}f(x)\nu_{n, 2}(\de x_2) \mu_{n, 1}(\de x) - \int_{\sX_1} \int_{\sX_2}f(x)\nu_{n, 2}(\de x_2) \nu_{n, 1}(\de x)\right\rvert\\
&\quad\leq \sup_{g \in \mathcal{F}(\sX_1)} \left\lvert\int_{\sX_1} g(x_1) \mu_{n, 1}(\de x_1) - \int_{\sX_1} g(x_1) \nu_{n, 1}(\de x_1)\right\rvert\\
&\quad= d(\mu_{n, 1}(\de x_1),\nu_{n, 1}(\de x_1)) \to 0,
\end{aligned}
\]
as $n \to \infty$ by assumption. Here we used that, since $\nu_{n,2}$ does not depend on $x_1$, the function $g(x_1)=\int_{\sX_2}f(x_1,x_2)\,\nu_{n,2}(\de x_2)$ belongs to $\mathcal{F}(\sX_1)$ whenever $f\in\mathcal{F}(\sX)$: this is immediate for $d=d_{\mathrm{TV}}$, while for $d=d_{\mathrm{BL}}$ we have $\|g\|_\infty\le 1$ and $|g(x_1)-g(x_1')|\le\int_{\sX_2}|f(x_1,x_2)-f(x_1',x_2)|\,\nu_{n,2}(\de x_2)\le\|x_1-x_1'\|_2$. Similarly we have
\[
\begin{aligned}
d(\mu_n,\rho_n) & = \sup_{f \in \mathcal{F}(\sX)} \left\lvert\int_\sX f(x) \mu_n(\de x) - \int_\sX f(x) \rho_n(\de x)\right\rvert\\
& \leq \int_{\sX_1} \left[ \sup_{g \in \mathcal{F}(\sX_2)} \left\lvert\int_{\sX_2} g(x_2) \mu_{n, 2}(\de x_2 | x_1) - \int_{\sX_2} g(x_2) \nu_{n, 2}(\de x_2)\right\rvert \right] \mu_{n, 1}(\de x_1)\\
& \leq \sup_{x_1} \, d(\mu_{n, 2}(\de x_2 | x_1),\nu_{n, 2}(\de x_2)) \to 0,
\end{aligned}
\]
as $n \to \infty$. The result then follows by \eqref{eq:triangle_TV}.
\end{proof}

We can now prove the remaining claims of Theorem \ref{thm:consistency-bvm}.

\begin{proof}[Proof of \eqref{active-BvM} and \eqref{boundary-gamma} in Theorem~\ref{thm:consistency-bvm}]
For each \(i=1,\ldots,k\), define
\[
T_{i,n}(x_i)
=
\bigl(\sqrt n(x_i^+-\hat x_i^+),\,n x_i^0\bigr),
\]
and let
\[
\Lambda_i
=
\mathcal N(\cdot;0,\Sigma_i^+)
\otimes
\bigotimes_{j=r_i+1}^{d}\Gamma(\alpha_j,1),
\]
with the convention that the second factor is a point mass at the empty
vector when \(r_i=d\). Since the joint conditional law of $X_{1:k}$ given $\data{1:k}$ is Markov and each $T_{i,n}$ is injective, the conditional law of $(T_{i,n}(X_i))_{i=1}^k$ is of the form required in Lemma~\ref{lm:decomposition_TV}, with kernels depending on the past only through $x_{i-1}$, while the target $\bigotimes_{i=1}^k\Lambda_i$ is a product. Hence, by Lemma~\ref{lm:decomposition_TV} (applied along almost surely convergent subsequences), it suffices to prove that
\[
d_{\mathrm{BL}}
\left(
\mathcal L_{\post_{1\given 1:k}(\cdot\mid\data{1:k})}
\bigl(T_{1,n}(X_1)\bigr),
\Lambda_1
\right)
\to0
\]
and, for every \(i=2,\ldots,k\),
\[
\sup_{x_{i-1}\in\Delta_d}
d_{\mathrm{BL}}
\left(
\mathcal L_{\tilde{\post}_{i\given 1:k}(\cdot\mid x_{i-1},\data{1:k})}
\bigl(T_{i,n}(X_i)\bigr),
\Lambda_i
\right)
\to0
\]
in \(\mu_{1:k}^{(\infty)}\)-probability, where \(d_{\mathrm{BL}}\) denotes the
bounded-Lipschitz metric between probability distributions. We prove the second
display; the argument for \(i=1\) is simpler. Fix \(i=2,\ldots,k\). By the HMM structure,
\[
\tilde{\post}_{i\given 1:k}(\de x_i\mid x_{i-1},\data{1:k})
=
\int_{\Delta_d}
\tilde{\post}_{i\given i-1,i+1}
(\de x_i\mid x_{i-1},x',\data{i})
r_{i+1\given 1:k}(x'\mid x_{i-1},\data{1:k})\,\de x' .
\]
Therefore, by convexity of \(d_{\mathrm{BL}}\),
\[
\begin{aligned}
&\sup_{x_{i-1}\in\Delta_d}
d_{\mathrm{BL}}
\left(
\mathcal L_{\tilde{\post}_{i\given 1:k}(\cdot\mid x_{i-1},\data{1:k})}
\bigl(T_{i,n}(X_i)\bigr),
\Lambda_i
\right)\\
&\qquad\le
\sup_{(x_{i-1},x_{i+1})\in\Delta_d^2}
d_{\mathrm{BL}}
\left(
\mathcal L_{\tilde{\post}_{i\given i-1,i+1}
(\cdot\mid x_{i-1},x_{i+1},\data{i})}
\bigl(T_{i,n}(X_i)\bigr),
\Lambda_i
\right).
\end{aligned}
\]
Hence it suffices to prove that the right-hand side converges to zero in
\(\mu_{1:k}^{(\infty)}\)-probability. Fix \(\epsilon>0\). By Lemma~\ref{lem:smoother-remainder}, there exists \(M>0\), independent of \(x_{i-1}\) and \(x_{i+1}\), such that 
\[
\sum_{|\kk|>M}\sum_{|\kk'|\ge 0}
\tilde w_{\nn_i,\kk,\kk'}(t_i - t_{i-1},t_{i+1} - t_{i})
+
\sum_{|\kk'|>M}\sum_{|\kk|\ge 0}
\tilde w_{\nn_i,\kk,\kk'}(t_i - t_{i-1},t_{i+1} - t_{i})
\le \frac{\epsilon}{6}.
\]
Therefore, by convexity of $d_{\text{BL}}$, and since $d_{\text{BL}}$  is bounded by $2$,
\[
\begin{aligned}
&\sup_{(x_{i-1},x_{i+1})\in\Delta_d^2}
d_{\mathrm{BL}}
\left(
\mathcal L_{\tilde{\post}_{i\given i-1,i+1}
(\cdot\mid x_{i-1},x_{i+1},\data{i})}
\bigl(T_{i,n}(X_i)\big),
\Lambda_i
\right)\\
&\quad\le
\sup_{(x_{i-1},x_{i+1})\in\Delta_d^2}
\sum_{\kk,\kk'\in\integer_+^d}
\tilde w_{\nn_i,\kk,\kk'}
(t_i-t_{i-1},t_{i+1}-t_i)
d_{\mathrm{BL}}
\left(
\mathcal L_{\pi_{\alpha+\nn_i+\kk+\kk'}}
\big(T_{i,n}(X_i)\big),
\Lambda_i
\right)\\
&\quad\le
\sup_{(x_{i-1},x_{i+1})\in\Delta_d^2}
\sum_{|\kk|\le M}
\sum_{|\kk'|\le M}
\tilde w_{\nn_i,\kk,\kk'}
(t_i-t_{i-1},t_{i+1}-t_i)
d_{\mathrm{BL}}
\left(
\mathcal L_{\pi_{\alpha+\nn_i+\kk+\kk'}}
\big(T_{i,n}(X_i)\big),
\Lambda_i
\right)
+
\frac{\epsilon}{3},
\end{aligned}
\]

We proceed separating the terms with no inactive shift (that is, for which we have $k_j = k_j'=0$ for all $j=r_i+1,\dots,d$) from the others. For the former group, an application of Lemma \ref{lem:boundary-dirichlet}, applied with \(\eta=\kk+\kk'\), gives
\[
\mathcal L_{\pi_{\alpha+\nn_i+\kk+\kk'}}
\bigl(T_{i,n}(X_i)\bigr)
\Rightarrow
N(\cdot;0,\Sigma_i^+)
\otimes
\bigotimes_{j=r_i+1}^{d}
\Gamma(\alpha_j,1)
=
\Lambda_i.
\]
The restriction of the sum in the second to last display over such $\kk$ and $\kk'$ is then seen to be smaller than $\varepsilon/3$ whenever $n$ is sufficiently large. We conclude considering
\begin{equation}
\label{Eq:Split}
\begin{split}
	&\sum_{\substack{|\kk|\le M,\ |\kk'|\le M\\
k_j+k_j'>0\ \text{for some }j>r_i}}
\tilde w_{\nn_i,\kk,\kk'}
(t_i-t_{i-1},t_{i+1}-t_i)
d_{\mathrm{BL}}
\left(
\mathcal L_{\pi_{\alpha+\nn_i+\kk+\kk'}}
\big(T_{i,n}(X_i)\big),
\Lambda_i
\right)\\
&\le
2\sum_{\substack{|\kk|\le M,\ |\kk'|\le M\\
k_j+k_j'>0\ \text{for some }j>r_i}}
\tilde w_{\nn_i,\kk,\kk'}
(t_i-t_{i-1},t_{i+1}-t_i)
\end{split}
\end{equation}
For all $|\kk|\le M,|\kk'|\le M$ satisfying $k_j + k_j' >0$ for some $j=r_i+1,\dots,d$,
set
\[
m_i(\kk,\kk')
=
\sum_{j=r_i+1}^{d}(k_j+k'_j) > 0.
\]
From \eqref{def_C_n}, for fixed $\kk,\kk'$, we see that the
Dirichlet--multinomial factor \(C_{\nn_i,\kk,\kk'}\) in
\(\tilde w_{\nn_i,\kk,\kk'}
(t_i-t_{i-1},t_{i+1}-t_i)\) depends on $n$ only through
\[
\frac{\Gamma(\theta+n)}
{\Gamma(\theta+n+|\kk|+|\kk'|)}
\prod_{j=1}^{d}
\frac{\Gamma(\alpha_j+n_{i,j}+k_j+k'_j)}
{\Gamma(\alpha_j+n_{i,j})}.
\]
By the standard gamma-ratio asymptotics,
$$
	\frac{\Gamma(\theta+n)}
	{\Gamma(\theta+n+|\kk|+|\kk'|)}
	=O\left( n^{-|\kk|-|\kk'|}\right).
$$
Also, for all \(j\le r_i\), since \(n_{i,j}/n\to x_{i,j}^*>0\),
$$
\frac{\Gamma(\alpha_j+n_{i,j}+k_j+k'_j)}
{\Gamma(\alpha_j+n_{i,j})}
=
O\left(n^{k_j+k'_j}\right).
$$ 
For inactive coordinates \(j>r_i\), we have \(n_{i,j}=0\), and the above gamma ratios are \(O(1)\). It follows that
\[
C_{\nn_i,\kk,\kk'}
=
O\left(
n^{-|\kk|-|\kk'|}
n^{\sum_{j=1}^{r_i}(k_j+k'_j)}
\right)
=
O\left(n^{-m_i(\kk,\kk')}\right) = o(1),
\]
since $m_i(\kk,\kk') >0$. Consequently, the unnormalised weights $r_{\nn,\kk,\kk'}$ in \eqref{def_C_n} also vanish asymptotically. Upon noting that the normalising constants in the definition of $\tilde w_{\nn_i,\kk,\kk'}(t_i-t_{i-1},t_{i+1}-t_i)$ are bounded below by $r_{\nn_i,0,0}$, which is strictly positive and independent of $n$, we conclude that also $\tilde w_{\nn_i,\kk,\kk'}(t_i-t_{i-1},t_{i+1}-t_i)$ vanishes asymptotically for all $|\kk|\le M,|\kk'|\le M$ satisfying $k_j + k_j' >0$ for some $j=r_i+1,\dots,d$. Thus, the finite sum \eqref{Eq:Split} is smaller than $\varepsilon/3$ when $n$ is large enough. Combined with the previously obtained bounds, this proves
\[
\sup_{(x_{i-1},x_{i+1})\in\Delta_d^2}
d_{\mathrm{BL}}
\left(
\mathcal L_{\tilde{\post}_{i\given i-1,i+1}
(\cdot\mid x_{i-1},x_{i+1},\data{i})}
\bigl(T_{i,n}(X_i)\bigr),
\Lambda_i
\right)
\to0
\]
in \(\mu_{1:k}^{(\infty)}\)-probability. The case \(i=1\) is analogous but uses the one-sided filtering mixture rather than the two-sided bridge mixture. Thus,
\[
\mathcal L
\left(
(T_{i,n}(X_i))_{i=1}^{k}
\mid \data{1:k}
\right)
\Rightarrow
\bigotimes_{i=1}^{k}\Lambda_i
\]
in \(\mu_{1:k}^{(\infty)}\)-probability. Recalling the definition of \(T_{i,n}\)
and \(\Lambda_i\), this is exactly the joint version of
\eqref{active-BvM} and \eqref{boundary-gamma}. The two marginal convergences
follow immediately.

Finally, if \(r_i=d\) for all \(i=1,\ldots,k\), the inactive blocks are empty, and the stronger conclusion in total variation is implied by the interior part of Lemma \ref{lem:boundary-dirichlet} (via the total variation part of Lemma \ref{lm:decomposition_TV}).
\end{proof}



\subsection{Proof of Proposition~\ref{cor:extended-locality}}\label{app:extended-locality}

We first need a preliminary lemma.
\begin{lemma}\label{lm:continuity_TV}
Let $p_s(x,x')$ and $\bridge{s}{t}{u}(z\mid x,x')$ be as in \eqref{eq:wf-transition-compact-sec2} and \eqref{eq:wf-bridge-mixture-sec2}. Then the mappings
\[
x \, \mapsto\, p_s(x,\cdot), \qquad (x, x') \, \mapsto \, \bridge{s}{t}{u}(\cdot \mid x,x')
\]
are continuous in total variation distance for every $s, t, u$.
\end{lemma}
\begin{proof}
Fix $\epsilon>0$. By Proposition~\ref{prop:coal-tail}, there exists $M>0$ such that
\[
\sum_{|\kk|>M}p_{\kk}^x(s)=\sum_{m>M}p_m(s)<\epsilon
\]
uniformly in $x\in\Delta_d$. Hence, for any $x,x'\in\Delta_d$,
\[
\begin{aligned}
d_{\mathrm{TV}}\!\big(p_s(x,\cdot),p_s(x',\cdot)\big)
&\le \sum_{|\kk|\le M}\bigl|p_{\kk}^x(s)-p_{\kk}^{x'}(s)\bigr|+4\epsilon.
\end{aligned}
\]
Moreover
\[
p_{\kk}^x(s)=p_{|\kk|}(s)\binom{|\kk|}{\kk}x^\kk
\]
 is continuous in $x$ for every fixed $\kk$. Therefore $x\mapsto p_s(x,\cdot)$ is continuous.
 
 As regards $\bridge{s}{t}{u}(\cdot \mid x,x')$, by Lemma \ref{lem:smoother-remainder} there exists $M > 0$ such that
 \[
\sum_{|\kk|>M}\sum_{|\kk'|\ge0}w_{\kk,\kk'}(s-t, u-s)
\;+\;
\sum_{|\kk'|>M}\sum_{|\kk|\ge0}w_{\kk,\kk'}(s-t, u-s)
<\epsilon,
\]
and continuity follows by a similar argument.
\end{proof}

\begin{proof}[Proof of Proposition~\ref{cor:extended-locality}]

We first prove \eqref{eq:filter-fullpast-tv}. Fix $\epsilon > 0$ and notice that by the Markov property we can write
\[
\post_{i\mid 1:i-1}(\de z\mid \data{1:i-1})
=\int_{\Delta_d}\trans{i}{i-1}(\de z\mid x)\,
\post_{i-1\mid 1:i-1}(\de x\mid \data{1:i-1}).
\]
By Lemma \ref{lm:continuity_TV} there exists $\delta>0$ such that
\[
d_{\mathrm{TV}}\!\big(\trans{i}{i-1}(\cdot\mid x),\trans{i}{i-1}(\cdot\mid x_{i-1}^*)\big)<\epsilon,
\]
whenever $\|x-x_{i-1}^*\|_2<\delta$, which implies
\[
\begin{aligned}
d_{\mathrm{TV}}\!\big(\post_{i\mid 1:i-1}(\cdot\mid \data{1:i-1}),
\trans{i}{i-1}(\cdot\mid x_{i-1}^*)\big)
&\le \int_{\Delta_d}
d_{\mathrm{TV}}\!\big(\trans{i}{i-1}(\cdot\mid x),\trans{i}{i-1}(\cdot\mid x_{i-1}^*)\big)
\post_{i-1\mid 1:i-1}(\de x\mid \data{1:i-1})\\
&\le \epsilon +\int_{A_{\delta,x_{i-1}^*}}\post_{i-1\mid 1:i-1}(\de x\mid \data{1:i-1}),
\end{aligned}
\]
with $A_{\delta,x}$ as in \eqref{def:set_A}, and the last term on the right hand side vanishes as $n \to \infty$ by \eqref{joint-consistency} in Theorem~\ref{thm:consistency-bvm}, applied to the truncated grid $t_1,\dots,t_{i-1}$, since $\post_{i-1\mid 1:i-1}(\cdot\mid\data{1:i-1})$ is the marginal at the last inspection time of the joint conditional law of $X_{1:i-1}$ given $\data{1:i-1}$. The result follows by arbitrariness of $\epsilon$.

For \eqref{eq:bridge-allother-tv}, write
\[
\Pi_i^{(n)}(\de x,\de x')
:=
\post_{i-1,i+1\mid 1:i-1,i+1:k}(\de x,\de x'\mid \data{1:i-1},\data{i+1:k}).
\]
By the Markov property, conditionally on $(X_{i-1},X_{i+1})=(x,x')$, the state $X_i$ is independent of $(\data{1:i-1},\data{i+1:k})$ with law $\bridge{i}{i-1}{i+1}(\cdot\mid x,x')$, so that
\[
\post_{i\mid 1:i-1,i+1:k}(\de z\mid \data{1:i-1},\data{i+1:k})
=\int_{\Delta_d}\int_{\Delta_d}
\bridge{i}{i-1}{i+1}(\de z\mid x,x')\,
\Pi_i^{(n)}(\de x,\de x').
\]
The signal observed on the reduced grid $t_1,\dots,t_{i-1},t_{i+1},\dots,t_k$ is again a WF-HMM, with transition $p_{h_i+h_{i+1}}$ between $t_{i-1}$ and $t_{i+1}$, and the data $(\data{1:i-1},\data{i+1:k})$ satisfy \eqref{eq:true_mechanism} on this grid. Applying \eqref{joint-consistency} to the reduced grid gives
\[
\Pi_i^{(n)}\Rightarrow \delta_{x_{i-1}^*}\times\delta_{x_{i+1}^*},
\]
$\mu_{1:i-1,i+1:k}^{(\infty)}$-almost surely. Fix $\epsilon>0$. By Lemma~\ref{lm:continuity_TV}, there exists $\delta>0$ such that
\[
d_{\mathrm{TV}}\!\big(\bridge{i}{i-1}{i+1}(\cdot\mid x,x'),
\bridge{i}{i-1}{i+1}(\cdot\mid x_{i-1}^*,x_{i+1}^*)\big)<\epsilon
\]
whenever $\|x-x_{i-1}^*\|_2+\|x'-x_{i+1}^*\|_2<\delta$. Therefore
\[
\begin{aligned}
d_{\mathrm{TV}}&\!\big(\post_{i\mid 1:i-1,i+1:k}(\cdot\mid \data{1:i-1},\data{i+1:k}),
\bridge{i}{i-1}{i+1}(\cdot\mid x_{i-1}^*,x_{i+1}^*)\big)\\
&\le \iint_{\Delta_d^2}
d_{\mathrm{TV}}\!\big(\bridge{i}{i-1}{i+1}(\cdot\mid x,x'),
\bridge{i}{i-1}{i+1}(\cdot\mid x_{i-1}^*,x_{i+1}^*)\big)\,
\Pi_i^{(n)}(\de x,\de x')\\
&\le \epsilon
+\Pi_i^{(n)}\!\Big(\bigl\{(x,x'):\|x-x_{i-1}^*\|_2+\|x'-x_{i+1}^*\|_2\ge\delta\bigr\}\Big),
\end{aligned}
\]
and the last term vanishes $\mu_{1:i-1,i+1:k}^{(\infty)}$-almost surely as $n\to\infty$, since the above set is closed and does not contain $(x_{i-1}^*,x_{i+1}^*)$. The result follows by arbitrariness of $\epsilon$.
\end{proof}



\section{Additional numerical illustrations}
\label{app:numerical-illustrations}

This section collects additional numerical results not included in the main paper. We first record additional law-level displays: a more moderate two-type predictive regime, the remaining predictive rate comparisons across inspection gaps, an inspection-time Gaussian comparison, and a three-type boundary regime. We then close with two tables, one summarizing sensitivity to $\aa$ and $\timegap$ and one assessing the finite-start approximation of the transition law.


\subsection{Two-type baseline regime}
\label{SM:two-type-baseline}

To complement the near-boundary regime in Figure~\ref{fig:predictive-transition-harsh} of the main paper, we consider a more moderate two-type setting with $\aa=(2,2)$, $x^*=0.15$, and inspection gap $\timegap=0.05$. Figure~\ref{fig:predictive-transition-baseline-supp} compares the exact predictive law with the transition law started at the empirical composition $\hat x$. The discrepancy is visibly smaller than in the near-boundary regime. The corresponding comparison of the induced Simpson-diversity distributions is shown in Figure~\ref{fig:diversity-twotype-main} of the main paper.

\begin{figure}[tbp]
\centering
\includegraphics[width=\textwidth]{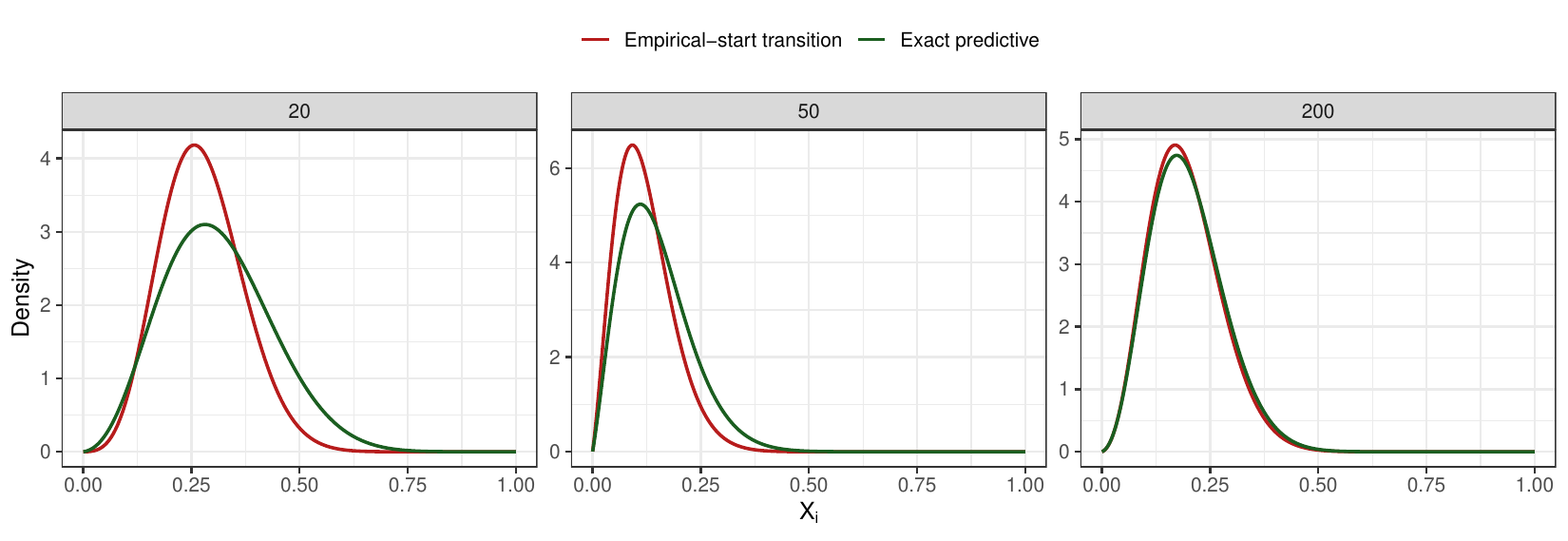}
\caption{Baseline two-type predictive comparison with $\aa=(2,2)$, $x^*=0.15$, $\timegap=0.05$, and $n\in\{20,50,200\}$. From left to right, the panels correspond to $n=20,50,200$ and compare the exact predictive density $\post_{i\mid i-1}(\cdot\mid \data{i-1})$ with the Wright--Fisher transition density $\trans{i}{i-1}(\cdot\mid \hat x_{i-1})$ started at the empirical composition.}
\label{fig:predictive-transition-baseline-supp}
\end{figure}


\subsection{Rate behavior across gaps}
\label{SM:rate-gaps}

To complement Figure~\ref{fig:rate-behavior-main} of the main paper,
which considers $\timegap=0.10$, Figure~\ref{fig:rate-behavior-gap-supp}
repeats the same experiment at the shorter gap $\timegap=0.05$ and the
longer gap $\timegap=0.50$. For each combination of $x^*$ and $n$, we
generate $10$ independent multinomial samples conditional on the fixed
true composition $x^*$.

\begin{figure}[tbp]
\centering
\includegraphics[width=\textwidth]{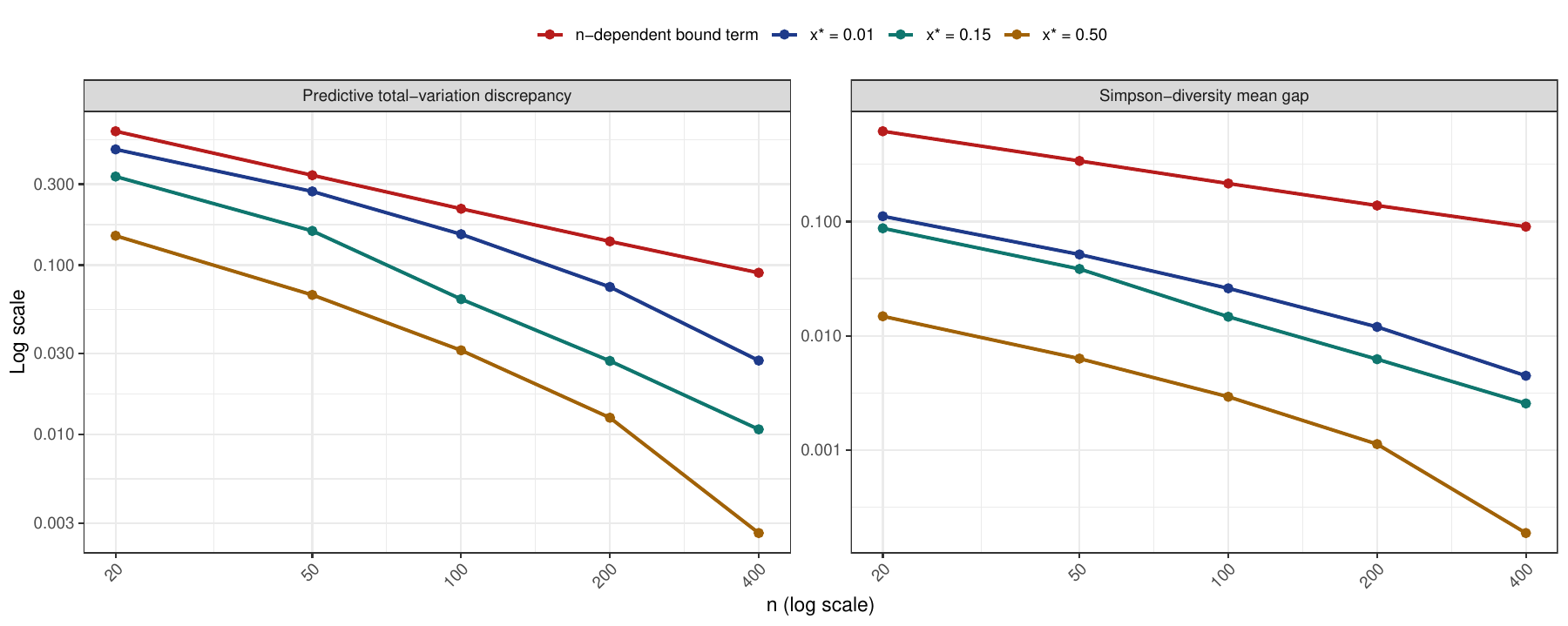}

\medskip

\includegraphics[width=\textwidth]{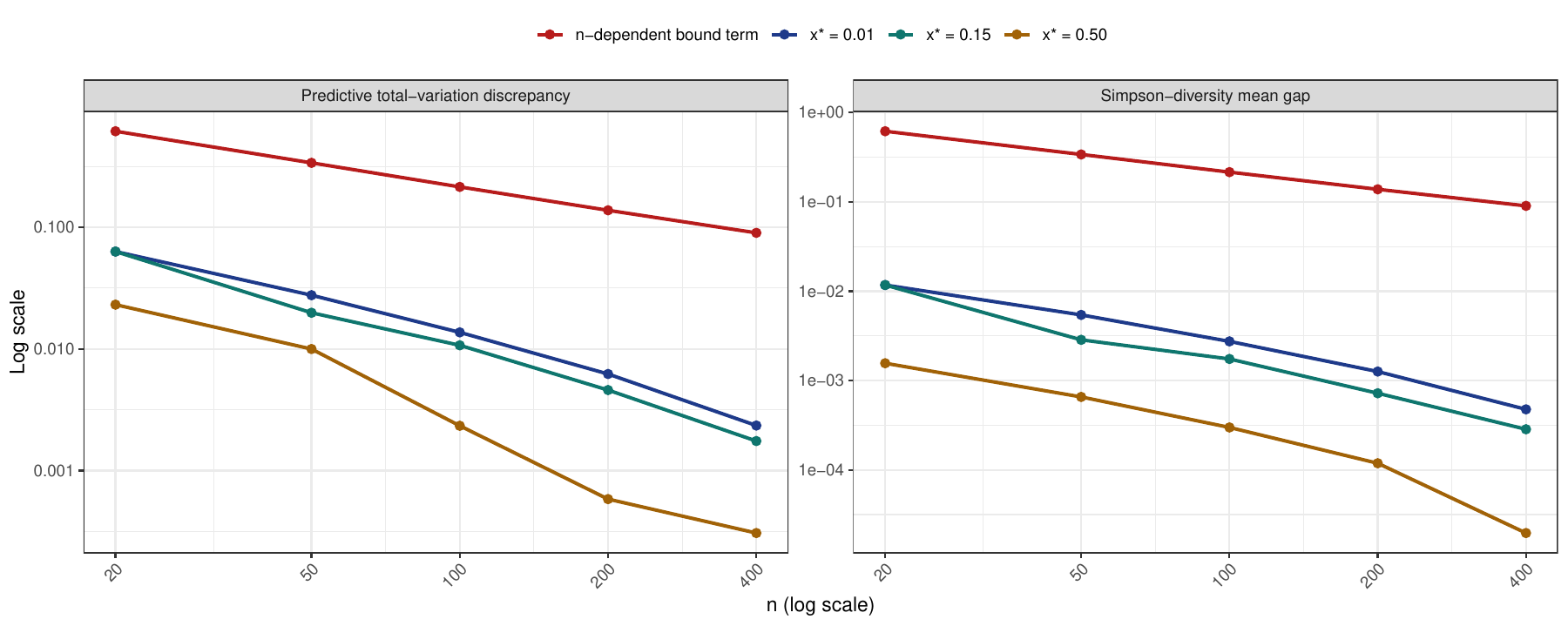}
\caption{Two-type predictive rate comparisons across
$x^*\in\{0.01,0.15,0.50\}$, with $\aa=(2,2)$ and
$n\in\{20,50,100,200,400\}$. The top row corresponds to
$\timegap=0.05$ and the bottom row to $\timegap=0.50$. For each
$(x^*,n)$ combination, the left panel shows the largest predictive
total-variation discrepancy among the $10$ repeated datasets, together
with the $n$-dependent factor in
Corollary~\ref{thm:predictive-empirical}. The right panel shows the
corresponding largest absolute discrepancy in the predictive mean of
Simpson diversity, together with the same reference term.}
\label{fig:rate-behavior-gap-supp}
\end{figure}


\subsection{Gaussian approximation at an inspection time}
\label{SM:gaussian-inspection}

We next illustrate the Gaussian component of the joint Bernstein--von Mises limit at an inspection time. We consider the two-type model with $\aa=(2,2)$, inspection gap $\timegap=0.05$, and latent path $(0.15,0.20,0.17)$ at three observation times. Figure~\ref{fig:smoother-gaussian-supp} compares the exact marginal smoother at the middle inspection time with the corresponding Gaussian approximation for increasing within-time sample sizes. The agreement improves visibly as $n$ increases.

\begin{figure}[tbp]
\centering
\includegraphics[width=\textwidth]{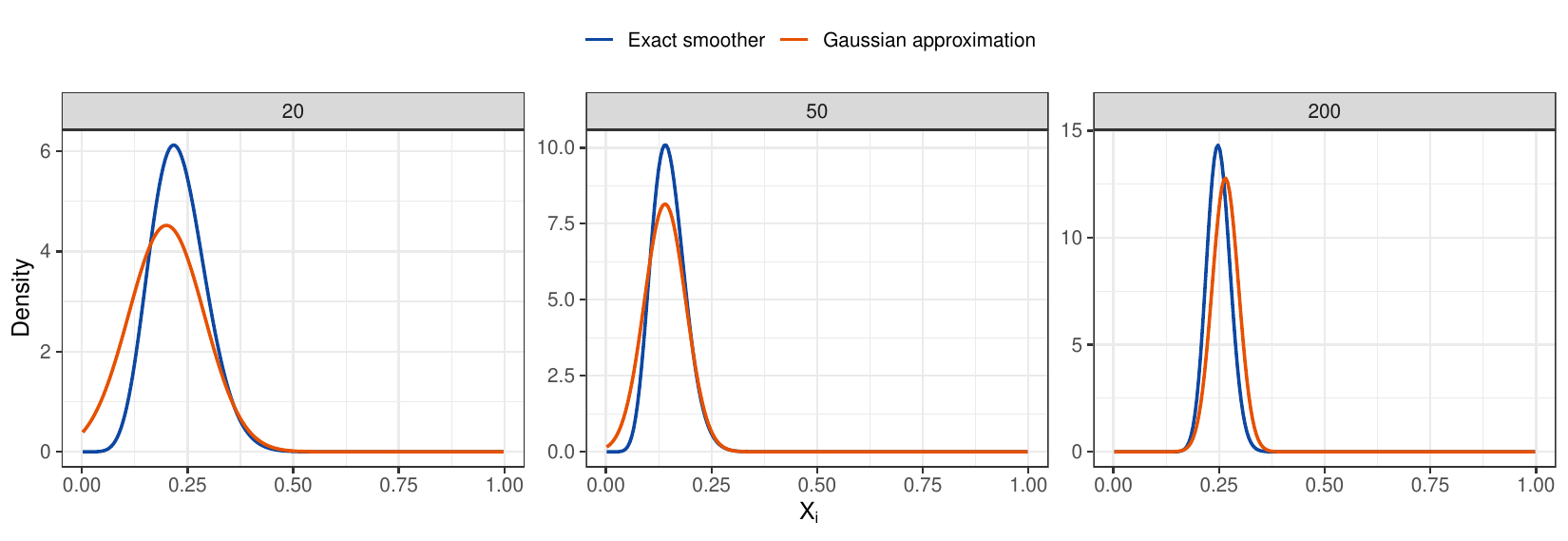}
\caption{Inspection-time Gaussian comparison in the two-type model with $\aa=(2,2)$, inspection gap $\timegap=0.05$, latent path $(0.15,0.20,0.17)$, and $n\in\{20,50,200\}$. From left to right, the panels correspond to $n=20,50,200$ and compare the exact marginal smoother $\post_{2\mid 1:3}(\cdot\mid \data{1:3})$ with the Gaussian law $\mathcal N(\cdot;\hat x_2,\Sigma_2)$ at the middle inspection time.}
\label{fig:smoother-gaussian-supp}
\end{figure}


\subsection{Multi-type case}
\label{SM:multitype}

We next complement the two-type predictive examples with a three-type
boundary regime. We take $\aa=(2,2,2)$, true frequency vector
$x^*=(0.60,0.40,0)$, and inspection gap $\timegap=0.05$, so that the third
coordinate is zero at the conditioning time and the observed empirical
composition has zero third component as well. Figure~\ref{fig:predictive-simplex-threetype-supp}
compares the exact predictive law with the empirical-start transition law
directly on the simplex. Figure~\ref{fig:diversity-threetype-supp}
gives the corresponding comparison of the induced Simpson-diversity
distributions. The displayed curves are kernel density estimates based on
$100{,}000$ Monte Carlo draws from each law.

\begin{figure}[tbp]
\centering
\includegraphics[width=\textwidth]{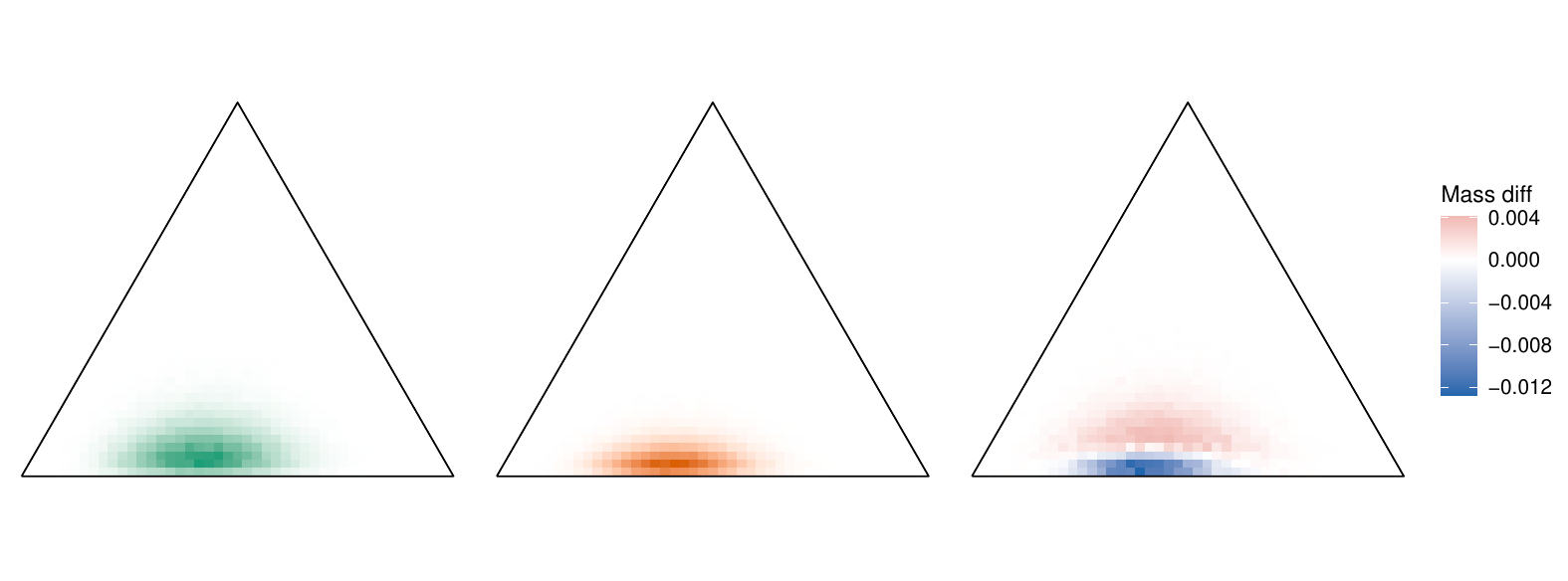}
\caption{Three-type predictive comparison in a boundary regime with $\aa=(2,2,2)$, $x^*=(0.60,0.40,0)$, $\timegap=0.05$, and $n=50$, shown on the simplex in barycentric coordinates. From left to right, the panels show the exact predictive law $\post_{i\mid i-1}(\cdot\mid \data{i-1})$, the empirical-start transition law $\trans{i}{i-1}(\cdot\mid \hat x_{i-1})$, and their binned mass difference (exact minus empirical-start transition).}
\label{fig:predictive-simplex-threetype-supp}
\end{figure}

\begin{figure}[tbp]
\centering
\includegraphics[width=\textwidth]{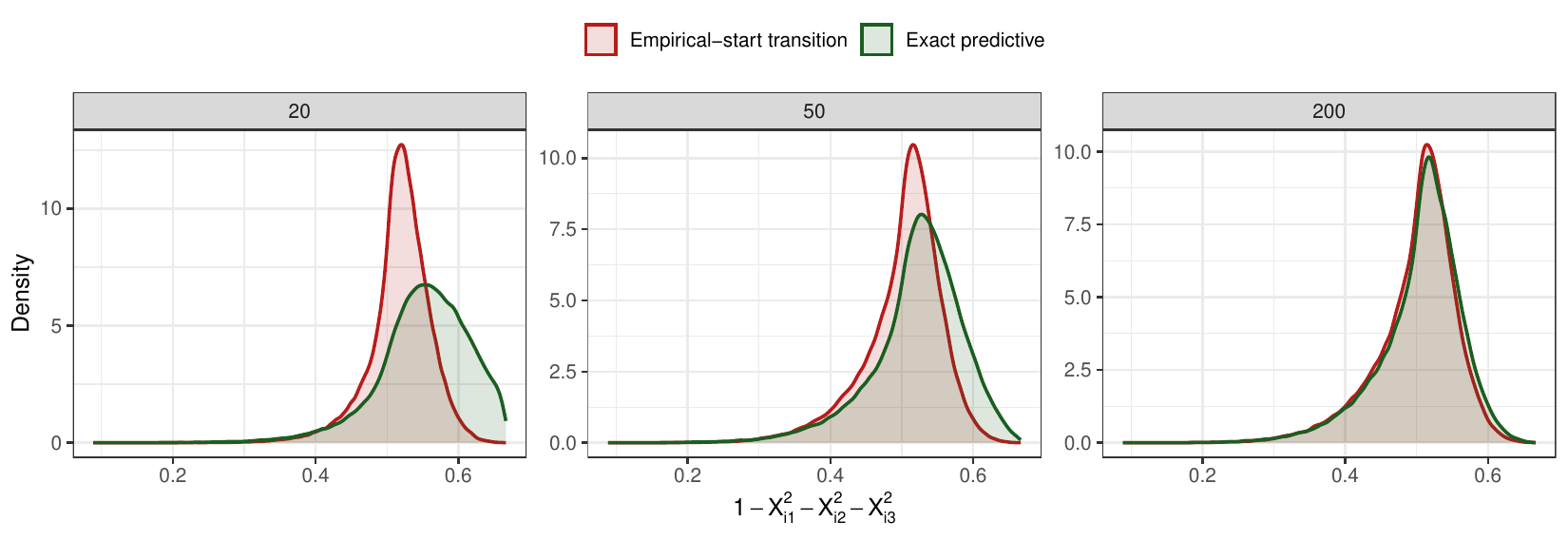}
\caption{Induced Simpson-diversity distributions of the exact predictive
law $\post_{i\mid i-1}(\cdot\mid \data{i-1})$ and the empirical-start
transition law $\trans{i}{i-1}(\cdot\mid \hat x_{i-1})$ in the same
three-type boundary regime, with $n\in\{20,50,200\}$. From left to right,
the panels correspond to $n=20,50,200$. The displayed curves are kernel
density estimates based on $100{,}000$ Monte Carlo draws from each law.}
\label{fig:diversity-threetype-supp}
\end{figure}


\subsection{Sensitivity to mutation and gap}
\label{SM:sensitivity}

To complement the rate comparison in Figure~\ref{fig:rate-behavior-main} of the main paper, Table~\ref{tab:regime-sensitivity-supp} examines sensitivity to both the mutation parameter and the inspection gap. We fix $n=50$ and $x^*=0.15$ and compare the symmetric mutation parameters $\aa=(0.5,0.5)$, $(2,2)$, and $(10,10)$ across three inspection gaps. The table reports the average and maximum predictive total-variation discrepancy together with the average absolute difference in the predictive probability that Simpson diversity exceeds $0.20$. The dependence on $\timegap$ is clear, with larger gaps producing smaller empirical discrepancies. The dependence on $\aa$ is less uniform, especially for short inspection gaps.

\begin{table}[tbp]
\centering
\small
\begin{tabular}{c|c|c|c}
\hline
$\aa$ & $\timegap$ & Average TV / Max TV & Average diversity-event gap \\
\hline
$(0.5,0.5)$ & 0.02 & 0.141 / 0.142 & 0.045 \\
$(2,2)$     & 0.02 & 0.184 / 0.286 & 0.073 \\
$(10,10)$   & 0.02 & 0.504 / 0.682 & 0.053 \\
\hline
$(0.5,0.5)$ & 0.05 & 0.071 / 0.071 & 0.022 \\
$(2,2)$     & 0.05 & 0.105 / 0.159 & 0.041 \\
$(10,10)$   & 0.05 & 0.284 / 0.417 & 0.004 \\
\hline
$(0.5,0.5)$ & 0.20 & 0.020 / 0.033 & 0.008 \\
$(2,2)$     & 0.20 & 0.044 / 0.067 & 0.020 \\
$(10,10)$   & 0.20 & 0.047 / 0.060 & 0.000 \\
\hline
\end{tabular}
\vspace{2mm}
\caption{Sensitivity to the mutation parameter and inspection gap at $n=50$ and $x^*=0.15$, based on $50$ repeated datasets. ``Average TV / Max TV'' gives the average and maximum predictive total-variation discrepancy. ``Average diversity gap'' gives the average absolute difference between the probabilities of the event $\{1-\sum\nolimits_{j=1}^2 X_{i,j}^2\ge 0.20\}$ under the exact predictive law and the empirical-start transition law.}
\label{tab:regime-sensitivity-supp}
\end{table}


\subsection{Finite-start approximation of the transition law}
\label{SM:coalescent-entrance}

Transition and bridge laws involving the coalescent representation are evaluated by approximating the coalescent entrance law with the block-counting process started from the finite level $1000$ throughout all numerical illustrations. Table~\ref{tab:nlarge-stability-supp} assesses the sensitivity of this choice by comparing transition laws obtained from initial levels $500$, $1000$, and $2000$ in three representative two-type regimes. In each case, the discrepancy between the approximations based on $1000$ and $2000$ is smaller than that between those based on $500$ and $1000$, and is below $0.005$. These comparisons indicate that the numerical results are only mildly sensitive to further increases beyond the starting level $1000$ used throughout.

\begin{table}[h!]
\centering
\begin{tabular}{c|c|c|c}
\hline
$x^*$ & $\timegap$ & TV $(500,1000)$ & TV $(1000,2000)$ \\
\hline
0.15 & 0.05 & 0.0088 & 0.0045 \\
0.25 & 0.10 & 0.0041 & 0.0021 \\
0.35 & 0.20 & 0.0017 & 0.0009 \\
\hline
\end{tabular}
\vspace{2mm}
\caption{Finite-start sensitivity of the transition law for the two-type model with $\aa=(2,2)$. The last two columns report total-variation distances between transition laws evaluated using coalescent initial levels $500$ and $1000$, and $1000$ and $2000$, respectively. All numerical illustrations use coalescent initial level $1000$.}
\label{tab:nlarge-stability-supp}
\end{table}

\end{document}